\documentclass{amsart}
\usepackage{amssymb,amsmath,amsfonts,amsthm}
\usepackage{graphicx}
\graphicspath{ {./images/} }
\usepackage{mathrsfs}
\usepackage{enumitem}
\usepackage{mathtools}

\usepackage{psfrag}

\usepackage{color}
\usepackage{todonotes}
\usepackage{enumitem}

\theoremstyle{plain}
\newtheorem{main}{Theorem}

\newtheorem{maincor}[main]{Corollary}
\newtheorem{theorem}{Theorem}[section]
\newtheorem{lemma}[theorem]{Lemma}
\newtheorem{proposition}[theorem]{Proposition}
\newtheorem{corollary}[theorem]{Corollary}
\theoremstyle{remark}
\newtheorem{remark}[theorem]{Remark}
\newtheorem{definition}{Definition}

\newtheorem{conjecture}{Conjecture}

\newtheorem{question}{Question}

\newcommand{\Leb}{\operatorname{vol}}

\newcommand{\Sing}{\operatorname{Sing}}

\newcommand{\diam}{\operatorname{diam}}

           \def\ea{\end{array}}
          \def\ec{\end{center}}
     \def\ed{\end{description}}
        \def\ee{\end{equation}}
       \def\eea{\end{eqnarray}}
     \def\eeaa{\end{eqnarray*}}
 \def\et{\end{thebibliography}}
\def\bib{\bibitem}

\def\Orb{{\rm Orb}}

\def\Cl{{\rm Cl}}

\def\Sing{{\rm Sing}}

\def\supp{\operatorname{supp}}

\def\cG{{\mathcal G}}

\def\cD{{\mathcal D}}
\def\cC{{\mathcal C}}
\def\cO{{\mathcal O}}

\def\cU{{\mathcal U}}

\def\sB{{\mathcal B}}

\def\cF{{\mathcal F}}

\def\cN{{\mathcal N}}
\def\cP{{\mathcal P}}
\def\cT{{\mathcal T}}

\def\length{\operatorname{length}}
\def\loc{{\operatorname{loc}}}

\def\vep{\varepsilon}

\def\ZZ{{\mathbb Z}}
\def\NN{{\mathbb N}}

\def\sA{{\mathscr A}}
\def\sB{{\mathscr B}}
\def\sC{{\mathscr C}}
\def\sD{{\mathscr D}}

\def\sX{{\mathscr X}}

\title[Entropy theory for singular flows]{A countable partition for singular flows, and its application on the entropy theory}

\author{Yi Shi, Fan Yang and Jiagang Yang}
\date{\today}
\thanks{ Y.S. is supported by NSFC 11701015, 11831001 and Young Elite Scientist Sponsorship Program by CAST. J.Y. is partially supported by CNPq, FAPERJ, PROEX-CAPES.}

\address{School of Mathematical Sciences, Peking Univesity, Beijing,  China.}
 \email{shiyi\@@math.pku.edu.cn}
\address{Department of Mathematics, University of Oklahoma, Norman, Oklahoma, USA.}
\email{fan.yang-2\@@ou.edu}
\address{
Departamento de Geometria, Instituto de Matem\'atica e Estat\'istica, Universidade
Federal Fluminense, Niter\'oi, Brazil.}
\address{Department of Mathematics, Southern University of Science and Technology, Guangdong, China.}
\email{yangjg\@@impa.br}

\begin{document}

\begin{abstract}
\end{abstract}

\begin{abstract}
	In this paper, we construct a countable partition $\sA$ for flows with hyperbolic singularities by introducing a new cross section at each singularity. Such partition forms a Kakutani tower in a neighborhood of the singularity, and  turns out to have finite metric entropy for every invariant probability measure. Moreover, each element of $\sA^\infty$ will stay in a scaled tubular neighborhood for arbitrarily long time.
	
	This new construction enables us to study entropy theory for singular flows away from homoclinic tangencies, and show that the entropy function is upper semi-continuous with respect to both invariant measures and the flows. 
\end{abstract}

\maketitle

\tableofcontents

\section{Introduction and statement of results}
\subsection{Singular flows}

{\em Singular flows} are flows that exhibit {\em equilibria}, or {\em singularities}. Such flows have been proven to be very resistant to rigorous mathematical analysis,
from both conceptual (existence of the equilibrium accumulated by regular orbits
prevents the flow to be hyperbolic) as well numerical (solutions slow down as
they pass near the equilibrium, which means unbounded return times and, thus,
unbounded integration errors and derivative) point of view.

For non-singular flows, the construction of {\em cross sections}, or {\em Poincar\'e sections}, is an important tool to study the dynamics of such flows, as it allows one to reduce the system to a discrete-time map (the {\em Poincar\'e map}) on the cross sections. See, for example, the celebrated work of Ratner~\cite{Ra} on Anosov flows, and the recent work by Lima and Sarig~\cite{LS} on three dimensional non-singular flows. However, the construction becomes far more difficult when the flow has a singularity. Often-times one has to construct several sections in order to capture  flow orbits that approach, and leave the singularity. 
See for example~\cite{APPV} and~\cite{GP}, where the authors constructed a family of cross sections for three-dimensional singular hyperbolic attractors, and~\cite{PT} for contracting Lorenz flows. Their construction requires a priori knowledge on how regular points approach singularities. Furthermore, they need linearization in a neighborhood of the singularity (thus putting assumptions on the eigenvalues of the tangent flow), require the stable foliation to have sufficient regular, and $\dim E^{cu}=2$ in order to reduce the dynamics on the cross sections to a one-dimensional system. Those assumptions significantly limit the situations where such strategy can be applied. As a result, as far as the authors are aware, there is no general construction of cross sections for singular flows on higher-dimensional manifolds.

In fact, the difficulty caused by the presence of equilibria shows up not only in the construction of cross sections, but also in the ergodic theory for flows.
It is a well accepted fact that for flows with singularities, the topological entropy, as well as metric entropies, can behave in a rather bizarre way. For example, in~\cite{SYZ} the authors constructed $C^\infty$ equivalent flows, such that one has zero entropy while the other has positive entropy. Even with those cross sections in~\cite{GP} and~\cite{PT}, the unbounded return time, which results in the unbounded derivative for the return map, has been proven to be the main obstruction for the ergodic theory of singular flows.

\subsection{Entropy theory for flows}

Entropy theory for flows not only is interesting by itself, but also has been proven to be a useful tool to classify the topological structure for flows.
In~\cite{PYY}, the authors use the entropy expansiveness to obtain a dichotomy on the chain recurrent classes of generic star flows, showing that every chain recurrent class with positive topological entropy must be isolated. More recently in~\cite{GYZ}, SRB-like measures (measures that are defined by Pesin's entropy formula) are used to classify the periodic orbit in the chain recurrent class for flows away from homoclinic tangencies.

However,
the entropy (both topological and measure-theoretical) for a flow is defined through its time-one map, whose dynamics is quite different from that of the Poincar\'e map. As a result, the cross sections constructed in the classical way (like those in~\cite{APPV}) does not work well for the entropy theory. Also  due to the difficult caused by singularities, there has been little development in the entropy theory of singular flows for many years. One of the recent breakthrough 
is the aforementioned work~\cite{PYY}, where it is proven that Lorenz-like flows are entropy expansive in any dimension. This, in particular, shows that the metric entropy is upper semi-continuous. However, the proof there strongly relies on the singularities being Lorenz-like and the entire flow being sectional hyperbolicity, therefore cannot be applied to singular flows in general.

\subsection{Statement of results: local dynamics near a hyperbolic singularity}

The goal of this paper is to give a complete description for the dynamics near a hyperbolic singularity $\sigma$, without making any extra assumption on the global structure of the flow itself. We will introduce a cross section $D_\sigma$ that contains the singularity,\footnote{Recall that in~\cite{APPV}, the cross sections are chosen away from the singularities.} and construct two countable measurable partitions, $\sC_\sigma$ and $\sA_\sigma$, using this cross section.

Below we let $X$ be a $C^1$ vector field and $\phi_t$ the associated flow on a Riemannian manifold $M$ without boundary. $\sigma\in\Sing(X)$ will be  a hyperbolic singularity of $X$. When we take a neighborhood of $\sigma$, we will always assume that $\sigma$ is the only singularity in this neighborhood. 

Given a neighborhood $B_r(\sigma)$ for a singularity $\sigma$, we will take the cross section to be:
\begin{equation}
D_\sigma = \exp_\sigma \left(\{v\in T_\sigma M: |v|\le\beta, |v^s| = |v^u|\}\right);
\end{equation}
One can think of it as the place where the flow speed is the ``slowest'', and orbit segments near $\sigma$ is ``making the turn''.
For each point $x\in D_\sigma$, we will write
$$
t^+_x = \inf\{\tau>0:\phi_\tau(x)\in\partial B_{r}(\sigma)\}
$$
and
$$
t^-_x = \inf\{\tau>0:\phi_{-\tau}(x)\in\partial B_{r}(\sigma)\}.
$$
for the first time that the orbit of $x$ exits the ball $B_r(\sigma)$ under the flow $\phi_t$ and $\phi_{-t}$.

Our first theorem is on the cross section $D_\sigma$ and the coarse partition $\sC_\sigma$, which gives an accurate estimate on how long each orbit spend in the neighborhood of $\sigma$. More importantly, despite $\sC$ being a countable partition, its metric entropy w.r.t. any invariant probability measure is uniformly bounded.

\begin{figure}
	\centering
	\def\svgwidth{\columnwidth}
	\includegraphics[scale = 1]{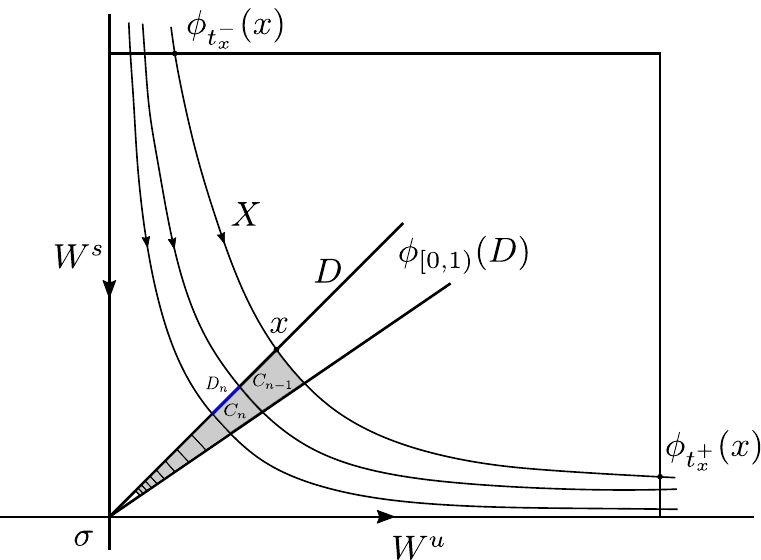}
	\caption{The partition $\sC_\sigma$.}
	\label{f.partitionC}
\end{figure}

\begin{main}\label{m.C}[The coarse partition $\sC_\sigma$]
	For every hyperbolic singularity $\sigma$ of a $C^1$ vector field $X\in\mathscr X^1(M)$ and every $r>0$ small enough, there is a cross section $D_\sigma\subset B_r(\sigma)$ containing $\sigma$ and a countable measurable partition $\sD_\sigma = \{D_n\}_{n>n_0(r)}$ on $D_\sigma$, with the following properties:
	\begin{enumerate}[label={(\Roman*).}]
		\item every orbit segment in $B_r(\sigma)$ intersects $D_\sigma$ only once;
		\item there is $0<L_0<L_1$ such that for every $n>n_0$ and $x\in D_n$, we have
		\begin{equation}\label{e.flowspeed}
		|X(x)|\in [L_0 e^{-n-1}, L_1 e^{-n}];
 		\end{equation}
 		\item there is $0<K_0<K_1$ such that for every $n>n_0$ and $x\in D_n$,
 		\begin{equation}\label{e.timeinB}
 		\frac{t^\pm_x}{n} \in [K_0, K_1];
 		\end{equation}
 		\item the closure of the set
		\begin{equation}\label{e.O(s)}
 		O(\sigma) = \bigcup_{n>n_0}\bigcup_{x\in D_n} \phi_{[-t^-_x,t^+_x]}(x)\subset B_{r}(\sigma)
 		\end{equation}
 		contains an open ball of $\sigma$ with diameter $\exp(-n_0)$;
 		\item the countable measurable partition $\sC_\sigma$ defined by:
 		$$
 		\sC_\sigma = \{C_n = \phi_{[0,1)}(D_n): n>n_0\}\cup\{(\cup_{n>n_0} C_n)^c
 		$$
 		forms a cone near $\sigma$, with $\sigma$ being the end point (see Figure~\ref{f.partitionC} and Figure~\ref{f.lorenz});
 		\item there exists $H_1>0$ such that for any probability  measure $\mu$, we have
 		\begin{equation}\label{e.entropyC}
 		H_\mu(\sC_\sigma)<H_1<\infty.
 		\end{equation}
	\end{enumerate}
	Furthermore, the above properties hold robustly in a $C^1$ neighborhood of $X$ and for the continuation of $\sigma$, with the same constants $L_0, L_1, K_0, K_1, H_1$.
\end{main}


Recall that for a diffeomorphism $f$ on a Riemannian manifold, the hyperbolicity (or the dominated splitting) of the tangent map $Df|_x$ determines the dynamics in a neighborhood of $x$ with uniform size. The same holds for non-singular flows. However, for flows with singularity, the situation is quite different: as discovered by Liao~\cite{Liao96}, the tangent flow governs the dynamics only in a {\em tubular neighborhood} along the orbit of $x$, and the size of this neighborhood is proportional to the flow speed. This means that near singularities, the size of such neighborhoods become much smaller (usually exponentially small if the singularity is hyperbolic), since the flow speed slows down exponentially. However, both the topological theory and the entropy theory for flows require estimates on a uniform size under the time-one map. This turns out to be the main obstruction for the study of singular flows.

To solve this issue, we will construct a countable partition $\sA_\sigma$, by taking any $L>0$ large enough and refining each element $C_n$ of $\sC_\sigma$ into $\cO(L^n)$ many elements, such that the diameter of each $A\in C_n$ is at most $\cO(L^{-n})$. Combine this with the estimation on the flow speed~\eqref{e.flowspeed}, we will show that each partition element of $\sA_\sigma$ controls the dynamics in a long tubular neighborhood. Moreover, the metric entropy of $\sA_\sigma$ is still uniformly bounded.

\begin{main}\label{m.A}[The refined partition $\sA_\sigma$]
	For every hyperbolic singularity $\sigma$ of a $C^1$ vector field $X\in\mathscr X^1(M)$ and  $r>0$ small enough, there exists $N_0>0$ such that for every $L\ge N_0$, $0<\beta<\beta_0$, there is a measurable partition $\sA_\sigma$ refining  $\sC_\sigma$, with the following properties:
	\begin{enumerate}[label={(\Roman*).}]
		\item for every $n>n_0 = n_0(r)$ and $C_n\in\sC_\sigma$, the collection $\sB_n : =\{A\in\sA_\sigma : A\subset C_n\}$ is a finite partition of $C_n$
		\item there exists $c_0>0$ independent of $L$, such that for  $L'=L^{K_1}e$ and every $n>n_0$, if $A\in\sA_\sigma$ satisfies $A\subset C_n$, then
		\begin{equation}\label{e.Adiameter}
		\diam(A)\le c_0\cdot \beta (L')^{-n};
		\end{equation}
		\item there exists $L''>0$ depending explicitly on $L$ and $c_1>0$, such that for every $n>n_0$, we have
		\begin{equation}\label{e.Bnumber}
		\#\sB_n\le c_1 (L'')^n;
		\end{equation}
		\item for two points $x,y\in A\in\sA_\sigma$, $y$ is in the  $\beta$-scaled tubular neighborhood of $x$ (for the precise definition, see the next section) until $x$ leaves $B_r(\sigma)$; 
		\item there exists $H_2>0$ depending on $L$, such that for any invariant probability  measure $\mu$, we have
		\begin{equation}\label{e.entropyA}
		H_\mu(\sA_\sigma)<H_2<\infty.
		\end{equation}
	\end{enumerate}
	Furthermore, the above properties hold robustly in a $C^1$ neighborhood of $X$ and for the continuation of $\sigma$, with the same constants $N_0, L', L'', c_0, c_1$ and $H_2$.
\end{main}

\begin{remark}\label{r.tower}
	For readers who are familiar with the language of {\em Rokhlin-Kakutani towers}, the set:
	$$
	\left\{\phi_j(A): A\in\sA, A\subset C_n \mbox{ for some }n>n_0, j \in [-K_1n, K_1n]\right\}
	$$
	forms a tower which contains the neighborhood $B_r(\sigma)$. The {\em base of the tower} is the cone
	$\Omega_0 = \phi_{[0,1)}(D_\sigma)$, which consists of elements of $\sA_\sigma$. Also note that our partition $\sA_\sigma$ treats the complement of the base $\Omega_0$ as a single element.
	
	Recall that in the classical definition of Rokhlin towers, the top floor is mapped back to the base of the tower. However, in our setting, the top floor of the tower is mapped to $(B_r(\sigma))^c$.
\end{remark}

\begin{figure}
	\centering
	\def\svgwidth{\columnwidth}
	\includegraphics[scale=0.9]{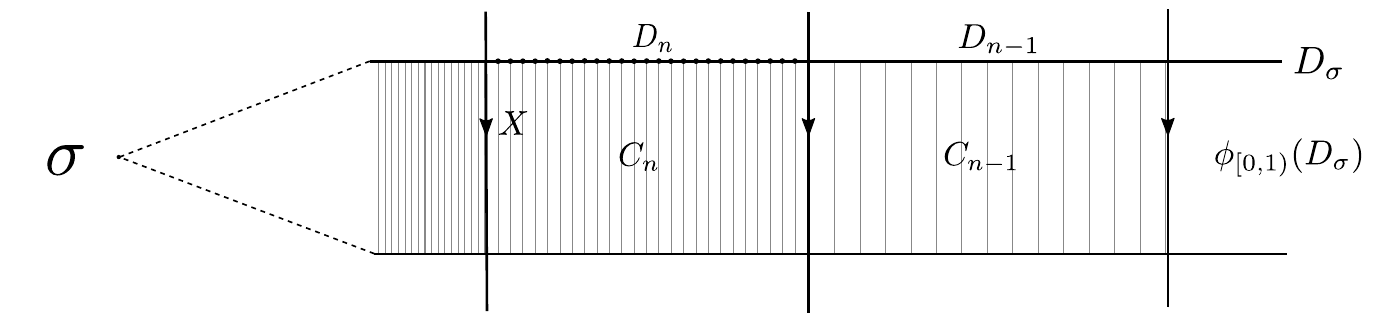}
	\caption{The partitions $\sB_n$ and $\sA_\sigma$.}
	\label{f.partitionB}
\end{figure}

\subsection{Statement of the result: when all the singularities are hyperbolic}
Observe that in the previous two theorems, we do not impose any hypothesis on other singularities of $X$; in fact, both theorems only deal with the local dynamics near $\sigma$. However, if one makes the assumption that all the singularities of $X$ are hyperbolic, then the construction of $\sA_\sigma$ can be carried out near each singularity. This leads to the next theorem:

\begin{main}\label{m.3}
	Let $X$ be a $C^1$ vector field, such that every singularity of $X$ is hyperbolic. Then for every $\beta>0$ small enough and $L\ge N_0$, there exists a countable, measurable partition $\sA$ with the following property:
	\begin{enumerate}[label={(\Roman*).}]
		\item for two points $x,y$ in the same element of the partition $\sA^\infty = \vee_{j\in\ZZ} \phi_j(\sA)$, the orbit of $y$ stays in the infinite $\beta$-scaled tubular neighborhood of $x$, under both $X$ and $-X$.
		\item there exists  $H>0$, such that for any invariant probability  measure $\mu$, we have
		\begin{equation}\label{e.entropy}
		H_\mu(\sA)<H<\infty.
		\end{equation}
	\end{enumerate}
	Furthermore, the partition $\sA$ can be made continuous for nearby $C^1$ vector fields, in the sense that if $X_n\xrightarrow{C^1} X$, then there is a sequence of partitions $\{\sA_n\}$ satisfying the above properties, such that for each element $A\in \sA$, there is $\{A_n: A_n\in\sA_n\}$ such that $\Cl (A_n)\to \Cl (A)$ in Hausdorff topology.
\end{main}

Regarding the notation: in this paper, a partition with an index, such as $\sA_\sigma$, $\sB_n$ and and $\sC_\sigma$, are constructed locally near the singularity $\sigma$; the partition $\sA$ without any index is defined for the entire flow. The only exception to this rule is in Section~\ref{s.8}, where we need to take a sequence of partitions $\sA_n$ and define a family of finite partitions $\sA_{n,N}$.

\begin{remark}\label{r.musupp}
	In all the theorems throughout this paper, unless otherwise specified, the measure $\mu$ may assign positive weight to some singularity $\sigma$. It is easy to check that this does not affect the estimation on the entropy $H_\mu$. See the proof of Proposition~\ref{p.Centropy} and~\ref{p.Aentropy} below.
\end{remark}

\begin{remark}
	Item (I) in Theorem~\ref{m.3} says that the set $\sA^\infty(x)$ is contained in the infinite scaled Bowen-ball of $x$ with size $\beta$. For the precise definition, see~\cite{WW}.
\end{remark}

\subsection{Applications: star flows, and flows away from homoclinic tangencies}
Next, we will state several applications for star flows, and for flows that are away from homoclinic tangencies. Recall that a vector field $X$ is said to be {\em star}, if there exists a $C^1$ neighborhood $\cU$ of $X$, such that for every $Y\in \cU$,  all the critical elements (singularities, periodic orbits) of $Y$ are hyperbolic.

\begin{maincor}\label{mc.star1}
	Let $X$ be a star vector field. Then for $L$ large enough, the partition $\sA$ given by Theorem~\ref{m.3} is ``almost'' generating, in the sense that for every ergodic, invariant probability measure $\mu$ and $\mu$-almost every $x\in M$, there exists $s(x)>0$ such that $\sA^\infty(x)$ is contained in the finite orbit segment $\phi_{(-s(x), s(x))}(x)$. In particular, for any ergodic, invariant probability measure $\mu$, we have
	$$
	h_\mu(X) = h_\mu(\phi_1, \sA).
	$$
\end{maincor}


Next, we turn our attention to flows away from tangencies, where the situation is more subtle. $X$ is said to exhibit {\em homoclinic tangency}, if $X$ has a hyperbolic periodic orbit with non-transverse homoclinic intersection. We denote by $\cT$ the collection of $C^1$ vector fields with homoclinic tangency.

Unfortunately, the partition $\sA$ may not be (almost) generating when the flow is away from tangencies but not star. However, we will prove that $\sA$ can be used to compute the metric entropy for any invariant measure $\mu$. For this purpose, we need the following theorem which generalizes the entropy expansiveness by Bowen~\cite{B72} and the almost entropy expansiveness in~\cite{LVY}.

Let $f: M\to M$ be a homeomorphism, $\mu$ an invariant probability such that $h_\mu(f)<\infty$.
Let $\sA$ be a measurable partition such that $H_\mu(\sA)<\infty$.

\begin{definition}\label{d.1}
	For any $x\in M$, we define its \emph{$\infty$ $\sA$-ball} by	
	$$\sA^\infty(x)=\bigcap_{n\in \mathbb{Z}} f^{-n} \sA(f^n(x)).$$
	We denote the \emph{$\sA$-tail entropy} of $x\in M$ by
	$$h_{tail}(f,x,\sA)=h_{top}(\sA^\infty(x),f).$$
	An invariant probability measure $\mu$ is called {\em $\sA$-expansive}, if for $\mu$ almost every $x$,
	$h_{tail}(f,x,\sA)=0$.
\end{definition}

\begin{main}\label{m.tailestimate}
Let $f$ be a homeomorphism over a compact manifold $M$ with finite dimension. Suppose $\mu$ is an invariant probability of $f$ that is $\sA$-expansive. Then we have
$h_\mu(f)=h_\mu(f,\sA)$.
\end{main}

A similar result holds for finite partitions where the tail entropy is defined by infinite Bowen-balls, see~\cite[THeorem 1.2]{CGY}. Note that if $\sA$ is a generating partition, then every measure is automatically $\sA$-expansive. On the other hand, if $f$ is $\vep$-entropy expansive, then for every partition $\sA$ with $\diam\sA<\vep$, every measure is $\sA$-expansive. We do not know if the converse is true.

\begin{question}
If $\sA$ is a finite or countable partition such that every invariant measure  $\mu$ is $\sA$-expansive and $H_\mu(\sA)<\infty$. Does this imply that
$$
h_{top}(\sA^\infty(x), f) = 0
$$
for {\em every} point $x\in M$? Does it imply $\vep$-entropy expansiveness for some $\vep>0$?
\end{question}

The following theorem generalizes Corollary~\ref{mc.star1} to flows away from tangencies:

\begin{main}\label{m.tangency}
	Let $X\in \sX^1(M)\setminus \Cl(\cT)$ be a $C^1$ vector fields with all the singularities hyperbolic. For $L\ge N_0$ and $\beta>0$ small enough, let $\sA$ be the partition given by Theorem~\ref{m.3}. Then every invariant probability measure $\mu$ is $\sA$-expansive. In particular, we have
	$$
	h_\mu(X) = h_\mu(\phi_1, \sA).
	$$
\end{main}

Furthermore, we obtain the upper semi-continuity for the metric entropy with respect to the flows in $C^1$ topology, and with respect to the measure in weak*-topology.

\begin{main}\label{m.continuous}
	For $X \in \sX^1(M)\setminus \Cl(\cT)$ with all singularities hyperbolic, there exists $L_2>0$ with the following property: if $X_n \in \sX^1(M)\setminus \Cl(\cT)$ is a sequence of  $C^1$ vector fields such that $X_n\xrightarrow{C^1}X$. Let $\mu_n,\mu$ be invariant measures of $X_n$ and $X$, respectively, with $\mu_n\to\mu$ in weak*-topology. Then we have
	$$
	\lim_{n\to\infty}h_{\mu_n}(X_n)\le h_\mu(X) + L_2\mu(\Sing(X)).
	$$	
	In particular, if $\mu(\Sing(X))=0$ then
	$$
	\lim_{n\to\infty}h_{\mu_n}(X_n)\le h_\mu(X).
	$$
	
\end{main}

This theorem shows that, if $\mu(\Sing(X))=0$ then $\mu$ is a point of upper-semi continuity for the metric entropy in the space $\{\mu: \mu \mbox{ is invariant for some } Y\in\cU\subset\sX^1(M)\}$.

Let us make some remarks on the condition $\mu(\Sing(X))=0$. It is proven in~\cite{LGW} that for diffeomorphisms away from tangencies, the metric entropy is upper semi-continuous. The proof requires one to consider a {\em finite} partition; for this purpose, it is natural to glue the tail of $\sA$ into a large element, and expect the entropy to remain approximately the same. However this is not the case for singular flows: we in fact prove that the loss of the metric entropy is (at most) proportional to the measure of a small region $O^N(\sigma)$ near the singularity; this region is, in fact, the image of $\cup_{n>N} C_n$ in $B_r(\sigma)$ under the flow. See Figure~\ref{f.finitepartition} and the precise statement in Theorem~\ref{t.finitepartitionentropy}.

This is possibly a new mechanism on how metric entropy is lost when a sequence of measures $\mu$ converges to $\mu$: if $\mu(\Sing(X))>0$, then we may lose metric entropy by (at most) a constant multiple of $\mu(\Sing(X))$. This phenomenon does not exist for non-singular flows or diffeomorphisms.

Also note that in~\cite{PYY} it is proven that the metric entropy is upper semi-continuous for Lorenz-like flows. However, the proof there comes from the entropy expansiveness, which relies on the sectional hyperbolic structure on the entire class. However, there are examples where this structure does not exist, even for star flows. See~\cite{BD}, an example of a chain recurrent class without dominated splitting.

On the other hand, the metric entropy for star flows may still be upper-semi continuous after all. This is due to the strong hyperbolicity of star systems. It is proven in~\cite{GSW} that every {\em Lyapunov stable} chain recurrent class of generic star vector field must be Lorenz-like, therefore entropy expansive (by~\cite{PYY}). It is also known from~\cite{PYY} that the example of Bonatti and da Luz~\cite{BD} is an isolated homoclinic class. This invites us to make the following conjectures:

\begin{conjecture}
	For (generic) star flows, the metric entropy varies upper-semi continuously.
\end{conjecture}

\begin{conjecture}
	There exists $X\in \sX^1(M)\setminus \Cl(\sX^*(M))$ and a sequence of invariant measures $\mu_n\to\mu$, along which the metric entropy is not upper semi-continuous. Here $\sX^*(M)$ is the space of $C^1$ star vector fields.
\end{conjecture}






\subsection{Structure of the paper}
The study of the local dynamics in $B_r(\sigma)$ is carried out in Section~\ref{s.3} and~\ref{s.4}.
Section~\ref{s.3} contains the construction of the coarse partition $\sC_\sigma$, and the proof of Theorem~\ref{m.C}, while Section~\ref{s.4} contains the construction of the refined partition $\sA_\sigma$, and the proof of Theorem~\ref{m.A}.
Then in Section~\ref{s.5} we combine the construction of $\sA_\sigma$ at each $\sigma\in\Sing(X)$ together and prove Theorem~\ref{m.3}.

The proof of Theorem~\ref{m.tailestimate} can be found in Section~\ref{s.6}. Then in Section~\ref{s.7} we show that the partition $\sA$ can be used to compute the metric entropy for star flows and flows away from tangencies, proving Corollary~\ref{mc.star1}, Theorem~\ref{m.tangency}. Finally we show the upper semi-continuity of the metric entropy in Section~\ref{s.8}.

We also include a comparison between our cross section $D_\sigma$ and those in~\cite{APPV} at the end of Section~\ref{s.3}.

\section{Preliminaries}\label{s.2}
Throughout this article, all the singularities of $X$ are assumed to be hyperbolic. Whenever we take a neighborhood $\cU$ of $X$ in $\sX^1(M)$, we will always assume that all the singularities of every $Y\in\cU$ are hyperbolic, and are exactly the continuation of those in $\Sing(X)$.

\subsection{The scaled linear Poincar\'e flow}

For a regular point $x$ and $v\in T_xM$, the {\em linear Poincar\'e flow} $\psi_t: \cN_x\to \cN_{\phi_t(x)}$ is the projection of $\Phi_t(v)$ to $\cN_{\phi_t(x)}$, where $\cN_x$ is the orthogonal complement of $X(x)$. To be more precise, we denote the normal bundle
of $\phi_t$ over $\Lambda$ by
$$
\cN_\Lambda = \bigcup_{x\in\Lambda\setminus\Sing(X)}\cN_x,
$$
where $\cN_x$ is the orthogonal complement of the flow direction $X(x)$, i.e.,
$$
\cN_x = \{v \in T_xM: v \perp X(x)\}.
$$
Denote the orthogonal projection of $T_xM$ to $\cN_x$ by $\pi_x$. Given $v \in \cN_x$ for a regular point $x \in
M \setminus  \Sing(X)$ and recall that $\Phi_t$ is the tangent flow,  we can define $\psi_t(v)$ as the  orthogonal projection of $\Phi_t(v)$ onto $\cN_{\phi_t(x)}$, i.e.,
$$
\psi_t(v) = \pi_{\phi_t(x)}(\Phi_t(v)) = \Phi_t(v) -\frac{< \Phi_t(v), X(\phi_t(x)) >}{\|X(\phi_t(x))\|^2}X(\phi_t(x)),
$$
where $<\cdot,\cdot >$ is the inner product on $T_xM$ given by the Riemannian metric.

The {\em scaled linear Poincar\'e flow}, which we denote by $\psi^*_t$, is defined as
\begin{equation}\label{e.scaledpoincare}
\psi^*_t(v) = \frac{\|X(x)\|}{\|X(\phi_t(x))\|}\psi_t(v) = \frac{\psi_t(v)}{\|\Phi_t|_{<X(x)>}\|}.
\end{equation}
It is introduced by Liao~\cite{Liao} to study flows with singularities.
Whenever necessary, we will write $\psi_{X,t}$ and $\psi^*_{X,t}$ to emphasis the dependence of $\psi$ and $\psi^*$ on the initial vector field $X$.

\begin{lemma}\label{l.C1norm}
For every $\tau>0$ and a $C^1$ neighborhood $\cU$ of $X$, there exists $L_{\tau,\cU}>0$ such that for every $Y\in\cU$ and $t\in[-\tau,\tau]$, we have
$$
\|\psi^*_{Y,t}\|\le L_{\tau,\cU}
$$
\end{lemma}

\begin{proof}
For fixed $\tau$ and $\cU$, $\Phi_{X,t}$ is uniformly bounded above and away from zero in both $X\in\cU$ and $t\in[-\tau,\tau]$.  As a result,  $\psi_{X,t}$ has uniformly bounded norm, so is $\psi^*$.
\end{proof}

\subsection{A scaled tubular neighborhood theorem}
For each regular point $x$ and $\beta>0$, we denote by $N_x(\beta)$ to be the submanifold given by
$$
N_x(\beta) = \exp_x(\cN_x(\beta)),
$$
where $\cN_x(\beta) = \{v\in \cN_x: |v|< \beta\}$, and $\exp_x$ is the exponential map from $T_xM$ to $M$. We may take $\beta_0>0$ small enough (but uniformly for $Y$ in a small $C^1$ neighborhood of $X$), such that $\exp_x$ is a diffeomorphism from $\cN_{Y,x}(\beta|X(x)|)$ to $N_{Y,x}(\beta|X(x)|)$ for every $\beta<\beta_0$. For such $\beta$, we define $\cP_{x,t}(y)$ to be the Poincar\'e map from $N_x(\beta)$ to $N_{\phi_t(x)}(\beta)$. In other words, $\cP_{x,t}(y)$ is the first point of intersection between the orbit of $y$, and the submanifold $N_{\phi_t(x)}(\beta)$. As before, $\cP_{X,x,t}(y)$ highlights the dependence of $\cP$ on the vector field.

For a regular point $x\in M\setminus \Sing(X)$ and $T>0$, $\beta>0$, we denote by
$$
B_{\beta} (x,T) = \bigcup_{t\in[0,T]} N_{X,\phi_t(x)}(\beta |X(\phi_t(x))|)
$$
to be the {\em $\beta$-scaled tubular neighborhood} of the orbit segment $\phi_{[0,T]}(x)$. We will refer to $T$ as the {\em length} of this tubular neighborhood. When $T=+\infty$, we call it {\em the infinite $\beta$-scaled tubular neighborhood}.

By continuity, $B_{\beta} (x,T)$ contains an open neighborhood of the orbit segment $\phi_{[\varepsilon,T-\varepsilon]}(x)$, for every $\vep>0$. One should note that the size of the neighborhood at $y\in  \phi_{[0,T]}(x)$ depends on the flow speed at $y$. Therefore, the neighborhood becomes smaller as the orbit gets closer to some singularity.

The next proposition provides a scaled tubular neighborhood theorem for flows with singularities. Most importantly, it gives a uniform size for the tubular neighborhood when normalizing with the flow speed. Such estimates played an important role in the work of Liao~\cite{Liao96} and~\cite{GY} on singular flows.

\begin{proposition}\label{p.tubular}\cite[Lemma 2.2]{GY}
There exists $L=L(X)>1$ and a small $C^1$ neighborhood $\cU$ of $X$, such that for every $\beta<\beta_0$, $Y\in\cU$ and every regular point $x$ of $Y$, $\cP_{Y,x,1}$ is well-defined and injective from $N_{Y,x}((\beta/L) |Y(x)|)$ to $N_{Y,\phi_{Y,1}(x)}(\beta|Y(\phi_{Y,1}(x))|)$. Moreover, for $y\in N_{Y,x}((\beta/L) |Y(x)|)$,  the orbit segment from $y$ to $\cP_{Y,x,1}(y)$ stays in the $\beta$-scaled tubular neighborhood of the orbit segment  $\phi_{[0,1]}(x)$.

\end{proposition}

\begin{proof}
	This is essentially Lemma 2.2 in~\cite{GY} with $T=1$. One only need to check that the constants there can be made uniform for $Y\in\cU$.
\end{proof}

For simplicity, below we will assume that $T>0$ is an integer. Apply the previous proposition recursively, we obtain the following proposition:

\begin{proposition}\label{p.tubular1}
	There exists $L=L(X)>1$ and a small $C^1$ neighborhood $\cU$ of $X$,  such that for every $T\in\NN$, $\beta<\beta_0$, $Y\in\cU$ and every regular point $x$ of $Y$, $\cP_{Y,x,T}$ is well defined and injective from $N_{Y,x}(\beta L^{-T} |Y(x)|)$ to $N_{Y,\phi^Y_T(x)}(\beta|Y(\phi_{Y,T}(x))|)$. Moreover,  the orbit segment from $y\in N_{Y,x}(\beta L^{-T} |Y(x)|)$ to $\cP_{Y,x,T}(y)$ stays in the $\beta$-scaled tubular neighborhood of the orbit segment  $\phi_{[0,T]}(x)$.
\end{proposition}

\begin{remark}\label{r.flytime}
	From the construction, we see that for $y\in N_{x}(\beta L^{-T} |X(x)|) $, if we denote by  $\tau_{x,T}(y)>0$ to be the first time that the orbit of $y$ hits the normal manifold $N_{\phi_T(x)}$, then Proposition~\ref{p.tubular} gives:
	$$
	\phi_{\tau_{x,1}(y)}(y) = \cP_{x,1}(y),
	$$
	
	In fact, more can be said regarding the hitting time $\tau_{x,T}(y)$. Note that in Proposition~\ref{p.tubular}, for any given $\varepsilon>0$, we can increase $L$ to obtain
	$$\tau_{x,1}(y)\le 1+\varepsilon\mbox{ for all }y\in N_{x}((\beta/L) |X(x)|).$$
	Then the recursive argument gives
	$$\tau_{x,T}(y)\le (1+\varepsilon)^T\mbox{ for all }y\in N_{x}(\beta L^{-T} |X(x)|).$$
\end{remark}

The map $\cP$ can be lifted to a map on the normal bundle using the exponential map in a natural way:
$$
P_{X,x,T} = \exp_{\phi_T(x)}^{-1}\circ\cP_{X,x,T}\circ\exp_x.
$$
The fact that the orbit segment in $N_{Y,x}(\beta L^{-T} |Y(x)|)$ remains in the scaled tubular neighborhood $B_\beta(x,T)$ guarantees that $P_{X,x,T}$ and $\cP_{X,x,T}$ are semi-conjugate by $\exp_{(\cdot)}$, and the previous proposition remains true for $P_{X,x,T}$ in $\cN_{X,x}( \beta L^{-T} |X(x)|)$ (in fact, this is how the scaled tubular neighborhood theorem is stated in~\cite{GY}). Furthermore, we have:

\begin{proposition}\cite[Lemma 2.3]{GY}\label{p.tubular2}
	$DP_{X,x,1}$ is uniformly continuous in the following sense: for every $\varepsilon>0$ and $\beta>0$, there exists $0<\delta<\beta L^{-1}$ such that for every $Y\in \cU$ and a regular point $x$ of $Y$, $y,y'\in N_{Y,x}(\beta L^{-1}|Y(x)|)$, if $d(y,y')<\delta L^{-1} |Y(x)|$, then we have
	$$
	|DP_{Y,x,1}(y) - DP_{Y,x,1}(y')|<\varepsilon.
	$$
	As a result, there exists $K>0$, independent of $Y\in\cU$ and $x\in M\setminus \Sing(Y)$, such that
	$$
	|DP_{Y,x,T}|\le K.
	$$
\end{proposition}

Note that the previous propositions remain valid if one replaces $L$ by a larger constant (we already used this fact in Remark~\ref{r.flytime}). This observation will play an important role in the construction of the sections near singularities.


\subsection{The entropy theory for countable partitions}

In his famous paper~\cite{Ma81}, Ma\~n\'e gave  a very useful criterion for a countable, measurable partition to have finite entropy:

\begin{lemma}\cite[Lemma 1]{Ma81}\label{l.finiteentropy}
	For every $N>0$,  there exists $H>0$ such that if $\sum_{n=1}^\infty x_n$ is a series with $x_n\in (0,1)$, such that $\sum_{n=1}^{\infty} n x_n<N$, then
	$$
	-\sum_{n=1}^\infty x_n\log x_n<H.
	$$
\end{lemma}

Note that the version that we stated here is slightly stronger than Ma\~n\'e's original statement; however, one can easily prove it using the same argument in~\cite{Ma81},

To put it in a more modern context, the previous lemma says that: if $\mu$ is a probability measure (not necessarily invariant) over a set $\Omega_0$, and $\Omega$ is a discrete time suspension over $\Omega_0$ with roof function $R:\Omega_0 \to \NN$ satisfying $\mu(R)<\infty$, then the partition of $\Omega_0$ into level sets of $R$:
$$
\sA = \{\Omega_k=R^{-1}(k):k\in\NN\}
$$
has finite entropy w.r.t. $\mu\mid_{\Omega_0}$. In this case, the suspension $\Omega$ can be seen as a Rokhlin-Kakutani tower over $\Omega_0$, and the lift of $\mu$ to $\Omega$ via
$$
\tilde\mu(A) := \sum_{k=1}^\infty \sum_{j=0}^{k-1} \mu(\Omega_k\cap T^{-j}(A))
$$
is a finite measure.

\section{Construction of the coarse partition $\sC_\sigma$}\label{s.3}

In this section, we will define the coarse partition $\sC_\sigma$. The construction consists of three steps:
\begin{enumerate}
	\item first, we will identify a cross section $D_\sigma$ in the neighborhood $B_r(\sigma)$. Unlike the previous construction in~\cite{APPV}, this new cross section, in fact, contains the singularity $\sigma$. One can think of it as the place where the flows speed is the slowest;
	\item we will cut $D_\sigma$ into countably many layers $\{D_n\}_{n>n_0}$, each of which is roughly $e^{-n}$ close to $\sigma$; we will show that the estimation on the flow speed~\eqref{e.flowspeed} holds on each $D_n$;
	\item finally, we define the partition $\sC_n$ in the cone $\phi_{[0,1)}(D_\sigma)$ by pushing each $D_n$ along the flow by time one; we will show that this partition has finite entropy for any probability measure.
\end{enumerate}
\subsection{The cross section $D_\sigma$}
To start with, we fix $\beta_1>0$ small enough, such that
\begin{itemize}
	\item for every $\sigma'\in\Sing(X), \sigma'\ne\sigma$, we have $B_{\beta_1}(\sigma)\cap B_{\beta_1}(\sigma')=\emptyset$;
	\item the exponential map $\exp_\sigma$ is well defined on $\{v\in T_\sigma{M}: |v|\le \beta_1\}$;
	\item the flow speed $|X(x)|$ is a Lipschitz function of $d(\sigma,x)$ on $\Cl(B_{\beta_1}(\sigma))$: there exists $0<L_0<L_1$, such that for every $\sigma\in\Sing(X)$ and every $x\in \Cl(B_{\beta_1}(\sigma))$, we have
	
	\begin{equation}\label{e.flowlip}
	\frac{|X(x)|}{d(x,\sigma)}\in[L_0,L_1].
	\end{equation}
	In particular, we have
	$|X(y)|\in [L_0 \beta_1,L_1\beta_1]$ for all $y\in\partial B_{\beta_1}(\sigma)$.
	\item the flow in $B_{\beta_1}(\sigma)$ is a $C^1$ small perturbation of the linear flow
	$$
	\tilde\phi_t(x) = e^{At}x,
	$$
	where $A$ is a matrix with non-zero eigenvalues;
	\item for $x\in B_{\beta_1}(\sigma)$, the tangent maps $D\phi_1(x)$ are small perturbations of the hyperbolic matrix $e^A$, with eigenvalues bounded away from $1$.
\end{itemize}
The second requirement is possible since the vector field $X$ is $C^1$, and the singularities are non-degenerate. Moreover, $L_0$ and $L_1$ can be made uniform in a $C^1$ neighborhood of $X$.

We treat the hyperbolic splitting $E^s\oplus E^u$ as orthogonal in $T_\sigma M$, in which we use the box norm. For $r\le \beta_1$, we will also think of $B_r(\sigma)$ as a box, whose sides are ``parallel'' to the stable and unstable manifolds of $\sigma$.\footnote{In fact, whether $B_r(\sigma)$ is a box or not does not affect our construction, as long as the flow speed on the boundary is bounded above and below.} For each $v\in T_\sigma M$, we write $v=(v^s, v^u)$ for the components of $v$ along $E^s$ and $E^u$, respectively.

We define $$\cD_\sigma(\beta_1) = \{v\in T_\sigma M: |v|\le\beta, |v^s| = |v^u|\},$$
and
\begin{equation}\label{e.D}
D_\sigma(\beta_1) = \exp_\sigma \left(\cD_\sigma(\beta_1)\right)
\end{equation}
for its projection to the manifold.

\subsection{The partition $\sC_\sigma$}

Below we will fix $r\le \beta_1$. We take $n_1$ large enough, such that $e^{-n_1}<r$ (below we will enlarge $n_1$ once to obtain $n_0$, see Lemma~\ref{l.tx}). For $n>n_1$, define
$$
D_n = D_\sigma(\beta_1) \cap \left(B_{e^{-n}}(\sigma)\setminus B_{e^{-(n+1)}}(\sigma)\right),
$$
and note that $D_n$ and $D_m$ are disjoint if $n\ne m$. Furthermore, \eqref{e.flowlip} immediate gives
\begin{equation}\label{e.speedDn}
|X(x)|\in[L_0e^{-(n+1)},L_1e^{-n}],
\end{equation}
as required by (II) of Theorem~\ref{m.C}.

Following~\cite[Section 5.3.1]{PYY}, for each $x\in B_{r}(\sigma)$, we write $x^s = d(x,W^u(x))$ and $x^u=d(x,W^s(\sigma))$, where $W^s(\sigma)$ and $W^u(\sigma)$ are the stable and unstable manifolds of $\sigma$, respectively. Then we define the {\em $\alpha$-cone on the manifold}, denote by $D^{i}_\alpha(\sigma)$, $i=s,u$, as:
$$
D^s_\alpha(\sigma) = \{x\in B_{r}(\sigma): x^u<\alpha x^s\},\hspace{1cm}D^u_\alpha(\sigma) = \{x\in B_{r}(\sigma): x^s<\alpha x^u\}.
$$
Clearly, the stable and unstable manifold of $\sigma$ are contained in the $\alpha$-cones, for all $\alpha>0$.
We also extend the hyperbolic splitting $E^s\oplus E^u$ on $T_\sigma M$ to $ B_{r}(\sigma)$ and define the $\alpha$-cone $\cC_\alpha(E^s)$ and $\cC_\alpha(E^u)$ on the tangent bundle. The next lemma shows that the cones on the manifold and the cones on the tangent bundle are naturally related.

\begin{lemma}\cite[Lemma 5.1]{PYY}\label{l.cones}
	There exists $K\ge1$, such that for all $\alpha>0$ small enough,
	\begin{enumerate}
		\item for every $x\in D^s_\alpha(\sigma)$, we have $X(x)\in \cC_{K\alpha}(E^s)$;
		\item for every $x\in B_{r}(\sigma)$, if $X(x)\in \cC_\alpha(E^s)$, we have $x\in D^s_{K\alpha}(\sigma)$.
	\end{enumerate}
	Moreover, the same holds for $D^u_\alpha(\sigma)$ and $\cC_\alpha(E^u)$.
\end{lemma}
The proof of this lemma easily follows from the fact that $\phi_t$ in $B_r(\sigma)$ is a small perturbation of the linear flow $e^{At}x$.
Note that for $x\in  B_{r}(\sigma)\setminus (D^s_\alpha(\sigma)\cup D^u_\alpha(\sigma))$, we lose  control on the direction of $X(x)$. One can think of the region  $  B_{r}(\sigma)\setminus (D^s_\alpha(\sigma)\cup D^u_\alpha(\sigma))$ to be the place where the flow is `making the turn' from the $E^s$ cone to the $E^u$ cone. The key observation is that, once $\alpha$ is fixed, the  time it takes from $D^s_\alpha(\sigma)$ to $D^u_\alpha(\sigma)$ is uniformly bounded. See~\cite[Lemma 5.2]{PYY}.

For each $n$ and $x\in D_n$, we write
$$
t^+_x = \inf\{\tau>0:\phi_\tau(x)\in\partial B_{r}(\sigma)\},
$$
and
$$
t^-_x = \inf\{\tau>0:\phi_{-\tau}(x)\in\partial B_{r}(\sigma)\}.
$$
The next lemma provide the estimate on $t^\pm_x$ for $x\in D_n$. which is the key for our construction.

\begin{lemma}\label{l.tx}
There exists constants $K_1>K_0>0$ independent of $r$, and $n_0\gg n_1$ depending on $r$, such that for each $n>n_0$ and $x\in D_n$, we have
$$
\frac{t^\pm_x}{n}\in[K_0,K_1].
$$
\end{lemma}

\begin{proof}

We only need to estimate $t^+_x$, then the same argument applied on the vector field $-X$ will give the desired result for $t^-_x$.

Since the flow in $B_{r}(\sigma)$ is a small perturbation of the flow $e^{At}$, we can take $\alpha_0>0$ small enough, such that the flow speed grows exponentially fast in $\cC_{K\alpha_0}(E^u)$, where $K>0$ is the constant given by Lemma~\ref{l.cones}. To be more precise, there exists $C,C'>0,1<\lambda<\lambda'$, such that for all $x\in D^u_{\alpha_0}(\sigma)$, we must have
\begin{equation}\label{e.exponential}
C\lambda^t\le\frac{|X(\phi_t(x))|}{|X(x)|}\le C'(\lambda')^t,
\end{equation}
provided that $\phi_{[0,t]}(x)\subset \Cl(B_r(\sigma))$.

We write
$$
t^u_x = \inf\{t>0:\phi_t(x)\in D^u_{\alpha_0}(\sigma)\}.
$$
By Lemma 5.2 of~\cite{PYY}, there exists $T^{\alpha_0}>0$, such that
\begin{equation*}\label{e.turn}
t^u_x<T^{\alpha_0} ,\mbox{ for all } x\in B_{r}(\sigma).
\end{equation*}
In particular, the above estimate holds uniformly on every $D_n$. As an immediate corollary, we get
$$
\frac{|X(\phi_{t^u_x}(x))|}{|X(x)|}\in[d_0,d_1]\mbox{ for some } d_1>d_0>0 \mbox{ independent of } n.
$$
Combine this with~\eqref{e.flowlip} and~\eqref{e.speedDn}, we see that
\begin{equation}\label{e.du}
|X(\phi_{t^u_x}(x))|\in [d_0L_0e^{-n-1},d_1L_1e^{-n}]\mbox{ for all } x\in D_n.
\end{equation}

We are left to control $t^+_x - t^u_x$ for $x\in D_n$. Note that for all $x\in D_\sigma$, we have
$$
d(\phi_{t^+_x}(x),\sigma)=r;
$$
consequently,
$$
|X(\phi_{t^+_x}(x))|\in[L_0r, L_1r].
$$
This combined with~\eqref{e.speedDn},~\eqref{e.exponential} and~\eqref{e.du} gives
\begin{equation*}
\frac{n+\log\frac{L_0r}{C'd_1L_1}}{\log\lambda'}\le t^+_x\le \frac{n+\log\frac{L_1r}{Cd_0L_0}}{\log\lambda}+T^{\alpha_0}.
\end{equation*}
In particular, there exists $n_0=n_0(r)$, such that if $n>n_0$ then we must have
$$
K_0 = \frac{1}{2\log\lambda'}\le \frac{t^+_x}{n} \le \frac{2}{\log\lambda}= K_1.
$$
We conclude the proof of this lemma.
\end{proof}
We make the following observation on the choice of the constants $K_0$ and $K_1$, which will be used in the next section.
\begin{remark}\label{r.uniformconstant}
	Note that the constants $L_0,L_1, C,C',d_0,d_1$ depends continuously on the vector field, thus can be made uniform in a $C^1$ neighborhood of $X$.  On the other hand, the constants $\lambda$, $\lambda'$ and $T^\alpha$ depends on the hyperbolicity of $\sigma$, therefore can be made uniform for nearby $C^1$ vector fields. 
\end{remark}

\begin{lemma}\label{l.intersect}
	Each orbit segment $\phi_{[-t^-_x,t^+_x]}(x)$ intersect with $ D_\sigma$ at exactly one point, which is $x$.
\end{lemma}

\begin{proof}
	For each $x\in B_{r}(\sigma)$, we write $\exp_\sigma^{-1}(x) = (v^s(x), v^u(x))$. Then this lemma easily follows from the fact that for $x\in D_\sigma$, $|v^s(\phi_t(x))|$ is strictly decreasing along the forward orbit of $x$, and $|v^u(\phi_t(x))|$ is strictly increasing (thanks to $\phi_t$ being a small perturbation of the linear flow $e^{At}$). Since points on $ D_\sigma$ satisfies $|v^s(x)|=|v^u(x)|$, $\phi_{[-t^-_x,t^+_x]}(x)$ and  $ D_\sigma$ can only intersect at $x$.
\end{proof}

The next lemma deals with the measure of the flow box $\phi_{[-t^-_x,t^+_x]}(D_n)$.

\begin{lemma}\label{l.measure.dn}
For every probability measure $\mu$ that is invariant under $\phi_t$, we have
$$\sum_{n>n_0}\mu\left(\bigcup_{x\in D_n}\phi_{[-t^-_x,t^+_x]}(x)\right)\le 1.$$
\end{lemma}
\begin{proof}
Recall that  $\{D_n\}$ are pairwise disjoint. Set
$$
\tilde D_n = \bigcup_{x\in D_n}\phi_{[-t^-_x,t^+_x]}(x),
$$
we claim that $\{\tilde D_n\}$ are also pairwise disjoint.

We prove by contradiction. Assume there exists $m\ne n$, $x\in D_m$, $y\in D_n$ such that
$$
\phi_t(x) = \phi_s(y),
$$
for $0<t<t_x^+$ and $0<s<t_y^+$. Then $x$ and $y$ are on the same orbit, which intersects with $ D_\sigma$ at two different points, a contradiction with the previous lemma.

As a result, we have
$$\sum_{n>n_0}\mu\left(\bigcup_{x\in D_n}\phi_{[t^-_x,t^+_x]}(x)\right)= \mu\left(\bigsqcup_{n>n_0}\tilde{ D}_n\right)\le\mu(B_{r}(\sigma))\le 1.$$
\end{proof}

From now on, we write, for each $n>n_0,$
$$
C_n = \phi_{[0,1)}D_n.
$$
Then $\cup_{n>n_0}C_n \subset \phi_{[0,1)} D_\sigma$ is contained in a  fundamental domain of  the time-one map $\phi_1$.
The next lemma shows that for $N>n_0$, the set $\cup_{n>N}C_n$ have uniformly small measure:

\begin{lemma}\label{l.measure.triangle}
	Let $K_0$ be the constant given by Lemma~\ref{l.tx}. For every $N>n_0$ and every invariant probability $\mu$, we have
	$$
	\mu\left(\bigcup_{n>N}C_n\right)\le \frac{1}{K_0N}.
	$$
\end{lemma}
\begin{proof}
	For every $x\in D_n$, $n\ge N$, Lemma~\ref{l.tx} gives
	$$
	t^+_x\ge K_0 n \ge K_0 N.
	$$
	Also note that for each $j,k\in\NN\cup\{0\}$ with $j\ne k\le \min_{x\in D_n}\{t^+_x\}$, the sets $\phi_{j}(C_n)$ and $\phi_{k}(C_n)$ are disjoint and have the same measure.
	Therefore, we get
	\begin{align*}
	N\cdot \mu\left(\bigcup_{n>N}C_n\right)=&\frac{1}{K_0}\mu\left(\bigcup_{n>N}\bigcup_{k=0}^{NK_0-1}\phi_k(C_n)\right)\\
	=&\frac{1}{K_0}\sum_{n>N}\mu\left(\phi_{[0,NK_0)}(D_n)\right)\\
	\le&\frac{1}{K_0}\sum_{n>N}\mu\left(\bigcup_{x\in D_n} \phi_{[0,t^+_x))} (x)\right) \le \frac{1}{K_0}.
	\end{align*}
\end{proof}

Now we are ready to construct the coarse partition $\sC_\sigma$. One can think of $\{C_n\}_{n>n_0}$ as a (one-sided) infinite cylinder, with the singularity $\sigma$ sitting at the end. See Figure~\ref{f.partitionB}.
We define:
$$
\sC_\sigma = \{C_n: n>n_0\}\cup\{M\setminus (\cup_{n>n_0}C_n)\}
$$
as a measurable, countable partition of $M$. See Figure~\ref{f.partitionC}.

\subsection{Finite entropy}

The next proposition states that the metric $\sC_\sigma$ w.r.t. any invariant measure is uniformly bounded.

\begin{proposition}\label{p.Centropy}
	There exists $H_1>0$, such that for every invariant probability measure $\mu$, 
	we have $$H_\mu(\sC_\sigma)<H_1<\infty.$$
\end{proposition}
\begin{proof}
	By Lemma~\ref{l.finiteentropy}, we only need to verify that $\sum_n n\mu(C_n)<N$ for some constant $N>0$. 
	
	By Lemma~\ref{l.tx}, we have $t^+_x \ge K_0 n$ on $D_n$. Then it follows that
	\begin{align*}
	n\mu(C_n) = &\frac{n}{K_0 n} \sum_{j=0}^{K_0n-1}\mu(C_n) = \frac{1}{K_0 } \sum_{j=0}^{K_0n-1}\mu(\phi_j(C_n))\\
	=&\frac{1}{K_0 }\mu\left(\bigsqcup_{j=0}^{K_0n-1}\phi_j (C_n)\right)
	\le \frac{1}{K_0}\mu\left(\bigcup_{x\in D_n}\phi_{[0,t^+_x]}(x)\right).
	\end{align*}
	Now we can sum over $n$ and obtain:
	\begin{align}\label{e.sumC}
	\sum_n n\mu(C_n)\le \sum_n \frac{1}{K_0}\mu\left(\bigcup_{x\in D_n}\phi_{[0,t^+_x]}(x)\right)\le \frac{1}{K_0},
	\end{align}
	where we apply Lemma~\ref{l.measure.dn} to obtain the last inequality.
\end{proof}

\begin{proof}[Proof of Theorem~\ref{m.C}]
	(I) is proved as Lemma~\ref{l.intersect}. (II) and (III) are obtained in Lemma~\ref{l.tx}. Item (IV) easily follows from the continuity of the flow and the choice of $n_0$. (V) is the definition of $\sC_\sigma$, and	(VI) is precisely Proposition~\ref{p.Centropy}. The proof of Theorem~\ref{m.C} is finished.
\end{proof}

\begin{remark}
	It is important to observe that the proof of Proposition~\ref{p.Centropy} does not depend on how small $\mu(B_r(\sigma))$ is, or whether $\mu(\sigma)=0$ or not. In fact, when $\mu(\sigma)>0$, one can obtain a better bound in both Lemma~\ref{l.measure.dn} and Proposition~\ref{p.Centropy}.
\end{remark}
\begin{remark}
	Note that the  construction of $\sC_\sigma$ depends continuously on the flow $X$ in $C^1$ topology. Furthermore, the constants in Lemma~\ref{l.measure.dn},~\ref{l.measure.triangle} and Proposition~\ref{p.Centropy} can be made uniform in a $C^1$ neighborhood.
\end{remark}

As can be seen from the proof of Lemma~\ref{l.tx}, $H_1$ can be chosen arbitrarily close to the largest Lyapunov exponent at $\sigma$ (by taking $n_0$ large enough). At first glance, this may seem to contradict with Ruelle's inequality; however, it is due to the fact that the partition $\sC_\sigma$ is not expansive in the sense of Definition~\ref{d.1}. This problem will be partially solved by the refined partition $\sA_\sigma$, which will be constructed in the next section.

\subsection{Comparison with the conventional sections}

Here we will relate our new section $D_\sigma$ to the cross sections $\Sigma^{i/o,\pm}$ constructed in~\cite{APPV}.

In~\cite{APPV} the authors considered {singularly  hyperbolic flows}, that is, flows on a three-dimensional manifold with an attractor $\Lambda$, on which there is a  dominated splitting $E^s\oplus E^{cu}$, such that the tangent flow on $E^s$ is uniformly contracting, and $E^{cu}$ is volume expanding. If $\Lambda$ is singular hyperbolic without any singularity, then it must be Anosov. On the other hand, if $\Lambda$ contains a singularity $\sigma$ that is accumulated by regular orbits, then $\sigma$ is {\em Lorenz-like}. Here being Lorenz-like means that the eigenvalues of $DX|_\sigma$ must satisfy
$$
\lambda_1>0>\lambda_2>\lambda_3, \mbox{ and } \lambda_1+\lambda_2>0.
$$
See~\cite{ArPa10} for more detail.

More importantly, it is proven that the strong stable manifold $W^{ss}(\sigma)$ (which is given by the dominated splitting $E^s\oplus E^{cu}$) is tangent to the eigenspace $E^3$ of $\lambda_3$, and intersects with $\Lambda$ only at the singularity $\sigma$. Combine this with~\cite{LGW} and~\cite{GSW}, we see that regular orbit in $\Lambda$ can only approach $\sigma$ in a very small cone around $W^{cu}(\sigma)$.\footnote{In fact, in a small cone around $E^2$. See the discussion in~\cite[Section 5.2]{PYY}.} Assuming linearization in a neighborhood of $\sigma$,\footnote{Note that this imposes certain conditions on the eigenvalue $\lambda_i$, $i=1,2,3$, especially if one requires the linearization to be sufficiently smooth.}  the authors constructed four cross sections, $\Sigma^{i/o,\pm}$, for each singularity. Here $\Sigma^{i,\pm}$ are used to capture orbits that approaches $\sigma$, and $\Sigma^{o,\pm}$ tracks those whose are leaving $\sigma$. See Figure~\ref{f.lorenz}. Using the smoothness of the  linearization, they show that the fly time $\tau$ from $\Sigma^i$ to $\Sigma^o$ satisfies
\begin{equation}\label{e.16}
\tau(x) = -\frac{\log x_1}{\lambda_1},
\end{equation}

where $x_1$ is the distance $x_1 = d(x,W^{cs}(\sigma)).$  In particular, $\tau$ is integrable w.r.t.\, the Lebesgue measure on $\Sigma^{i,+}$.

Now let us describe the relation between $D_\sigma$ and $\Sigma^{i/o}$. Assuming that $\sigma$ is a Lorenz-like singularity for some vector field $X$ on a three-dimensional manifold $M$ (and note that we do not need the singular-hyperbolicity outside a neighborhood of $\sigma$), we may further assume that $\Sigma^{i/o,\pm}$ are taken on the set $\partial B_r(\sigma)$. Here we can take $B_r(\sigma)$ to be a cube around $\sigma$ which does not affect our construction, as we only need $B_r(\sigma)$ to be small and the flow speed on $\partial B_r(\sigma)$ to be bounded above and below.

\begin{figure}
	\centering
	\def\svgwidth{\columnwidth}
	\includegraphics[scale=0.9]{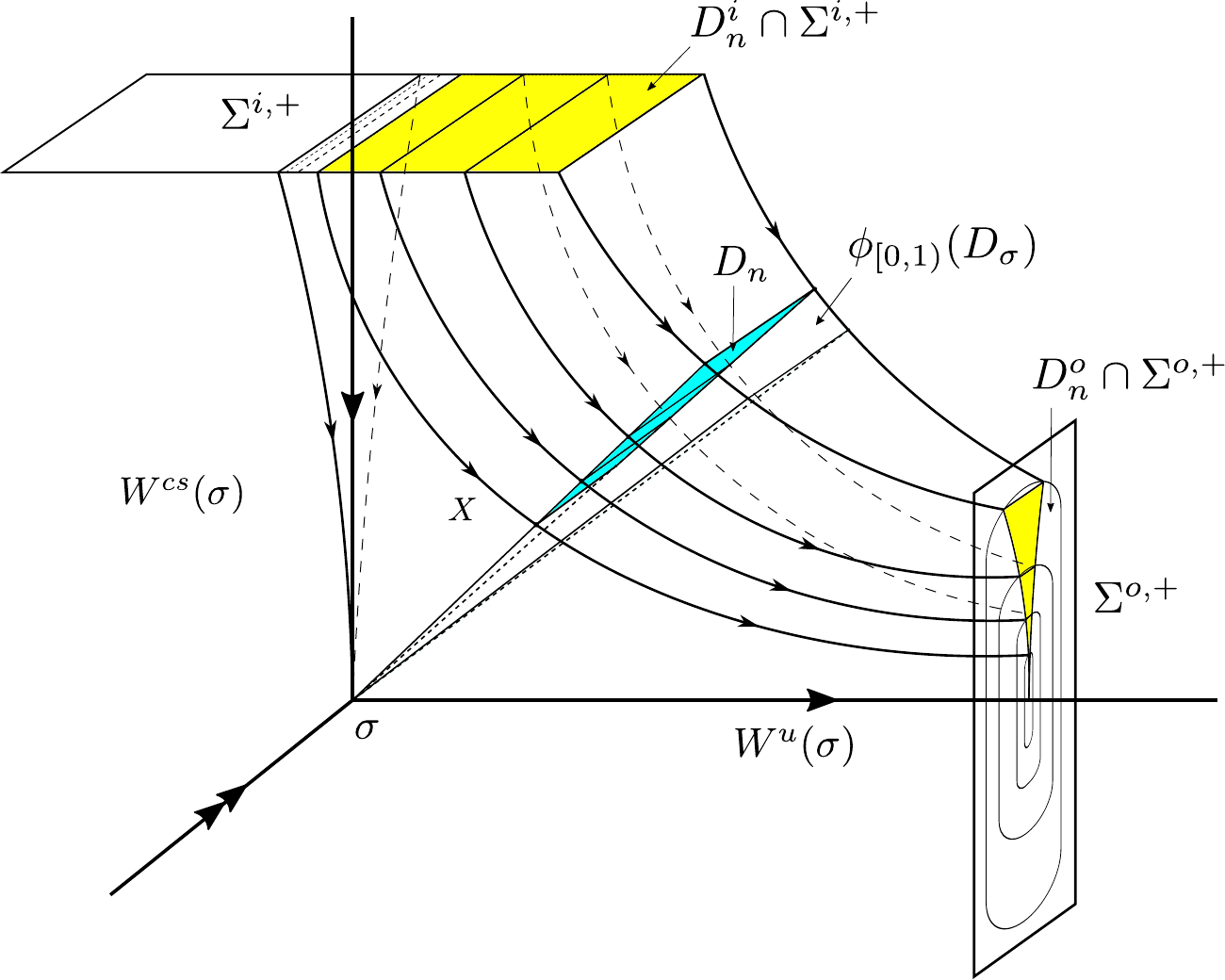}
	\caption{The image and pre-image of $\{D_n\}$ on the sections $\Sigma^{i/o,+}\subset \partial B_r(\sigma)$.}
	\label{f.lorenz}
\end{figure}

Then we can construct $D_\sigma$ and $\{D_n\}_{n>n_0}$ as before. For simplicity, we will only focus on the right half-space which contains $\Sigma^{o,+}$. If we define
$$
 D_n^o = \bigcup_{x\in D_n} \phi_{t^+_x}(x),\mbox{ and }D_n^i = \bigcup_{x\in D_n} \phi_{-t^-_x}(x)
$$
for the image and the pre-image of $D_n$ on $\partial B_r(\sigma)$ under the flow $\phi_t$, then $\{D_n^i\cap \Sigma^{i,+}\}$ is a family of countably many strips (the yellow strips in Figure~\ref{f.lorenz}) on $\Sigma^{i,+}$, which becomes closer to the curve $W^{cs}(\sigma)\cap\Sigma^{i,0}$ as $n$ gets larger. Their forward image under the flow:
$$
\phi_{\tau}(D_n^i\cap \Sigma^{i,+})
$$
are contained in a triangular region inside $D_n^o\cap\Sigma^{o,+}$.

Recall that Lemma~\ref{l.tx} shows $t_x^- = \cO(n)$ on $D_n$, which means that for $x\in D_n^i\cap \Sigma^{i,+}$, the distance $d(x,W^{cs}(\sigma)) $ is of order $e^{-(\lambda_1+1)n}$. Combined with~\eqref{e.16}, this shows that
\begin{equation*}
\tau(x) \approx \frac{\lambda_1+1}{\lambda_1}n\,\,\mbox{ for } x\in D_n^i\cap\Sigma^{i,+}.
\end{equation*}
In other words, the partition $\{D_n\}$ on $D_\sigma$ induces a countable partition $\{D_n^i\cap\Sigma^{i,+}\}$ on $\Sigma^{i,+}$, which can be seen as the level sets of $\tau$.

In a later work~\cite{GP}, the authors considered the return map $T$ on the cross sections $\Sigma^{i,+}$. They showed that the return map $T$ can be reduced to a one-dimensional, uniformly expanding map $T_{L}$ on the interval $[-1/2,1/2]$ with unbounded derivative at zero, known as the Lorenz-map. Then our partition $\{D_n\}$ naturally induces a countable partition on $[-1/2,1/2]$, which is of the form
$$
\left\{(-e^{-{(\lambda_1+1)n}},0), (0, e^{-(\lambda_1+1)n}): n>n_0\right\}.
$$
Note that partitions of this form has been widely used to study unimodal maps, namely interval maps with zero derivative at the point $0$.

Similarly, the same treatment can be applied to the contracting Lorenz-attractors in \cite{Rov},  resulting in the same partition $\left\{(-e^{-{(\lambda_1+1)n}},0), (0, e^{-(\lambda_1+1)n}): n>n_0\right\}$ for the one-dimensional Rovella maps. Such maps can be seen as unimodal maps with discontinuity at zero, and our partition coincide with the the classical partitions for the Rovella maps. See~\cite{PT} for more detail.

Finally, we would like to emphasis that, unlike in~\cite{APPV} and~\cite{GP}, our construction for the cross section $D_\sigma$ and the countable partition $\{D_n\}$:
\begin{enumerate}
	\item does not require knowledge on know regular orbits approaches $\sigma$;
	\item does not need information on the hyperbolicity of $X$ at regular points;
	\item avoids linearization altogether, thus does not require any condition on the eigenvalues at $\sigma$;
	\item can be applied in any dimension.
\end{enumerate}

In fact, our estimation in Lemma~\ref{l.tx} is enough to show that $\tau$ is integrable w.r.t.\,the Lebesgue measure on $\Sigma^{i,+}$, which is a crucial step in~\cite{APPV} ,~\cite{GP} and~\cite{PT}.

\section{Construction of the refined partition $\sA_\sigma$}\label{s.4}
As we have discussed before, both the topological theory (size of the invariant manifold, transverse homoclinic intersections, etc.) and the entropy theory (entropy expansiveness, upper semi-continuity of the metric entropy, etc.) requires estimation on a uniform size. However, for singular flows, the tangent map determines the underlying dynamics only in the scaled tubular neighborhood along the orbit. From Proposition~\ref{p.tubular1} which goes back to the classical work of Liao~\cite{Liao96}, the size of such neighborhoods depend on the length and the flow speed at each point. Combine this with Lemma~\ref{l.tx}, we see that for points in $C_n$, the size of such neighborhoods must be exponentially small.

This observation forces us to construct a new partition $\sA_{ \sigma}$ by refining each element of $\sC_\sigma$ with a finite partition $\sB_n$ (due to the observation above, the cardinality of $\sB_n$ must be exponential in $n$), such that on each element of $\sB_n$, the scaled tubular neighborhood is sufficiently long. For this purpose, we need a sharp control over the flow speed $|X(x)|$, and the time it takes for the point $x$ to leave $B_r(\sigma)$, which
is already given by Theorem~\ref{m.C}. The main difficulty here is to show that $\sA_\sigma$ still has finite entropy.

\subsection{The partition $\sB_n$ and $\sA_\sigma$}

Recall that $L_{1,\cU}$ is an upper bound of the scaled linear Poincar\'e flow $\psi_t^*$ given by Lemma~\ref{l.C1norm}. Let $L(X)$ and $L(-X)$ be the constants given by Proposition~\ref{p.tubular1} for the vector field $X$ and $-X$, respectively. We define
\begin{equation}\label{e.L}
N_0 =\max\{L(X), L(-X),L_{1,\cU}\},
\end{equation}
where $K_1$ and $K_0$ are the constants given by Lemma~\ref{l.tx}. For any given $L\ge N_0$ and $0<\beta<\beta_0$ with $\beta_0$ given by Proposition~\ref{p.tubular1}, we consider balls with center in $D_n$ and radius:
\begin{equation}\label{e.rn}
r_n := \beta L^{-K_1 n}L_0e^{-(n+1)},
\end{equation}
where $L_0$ is given by~\eqref{e.flowlip}.

Fix $r<\beta_1$, then Theorem~\ref{m.C} gives a countable partition $\sC_n = \{C_n\}_{n>n_0}$ in the neighborhood $B_r(\sigma)$. Also recall that each $C_n$ is the image of some $D_n\subset D_\sigma$ under the flow by time one. For each $n>n_0$, we take a $r_n$-separated set in $\overline{D_n}$ with maximal cardinality, denote by $E_n$. Here $E_n$ being $r_n$-separated means that for every $x,y\in E_n$, we have $d(x,y)> r_n$.
\begin{lemma}\label{l.diam}
	There exists a finite partition $\tilde  \sB_n$ of $D_n$, such that for every $B\in\tilde  \sB_n$, there exists $x\in E_n$ with $B_{r_n/2}(x)\cap D_n\subset B \subset B_{r_n}(x)\cap D_n$. In particular, we have
	$$
	\diam B \le \beta \cdot \frac{L_0 }{e}\left(L^{K_1}e\right)^{-n}  ,\,\forall B\in\tilde\sB_n
	$$
\end{lemma}

\begin{proof}
	Since $E_n$ is $r_n$-separated with maximal cardinality, we have
	$$
	B_{r_n/2}(x)\cap B_{r_n/2}(y)=\emptyset \mbox{ for } x,y\in E_n, x\ne y,
	$$
	and
	$$
	D_n\subset \bigcup_{x\in E_n}B_{r_n}(x).
	$$
	Furthermore, the same hold when restricted to $D_n$. 	Then the existence of such partition immediately follows from the above properties.
\end{proof}


The choice of $L$ in~\eqref{e.L} together with Proposition~\ref{p.tubular1} lead to the following proposition:

\begin{proposition}\label{p.tubular.sing}
	For the partition $\{ \tilde\sB_n\}$ given by the previous lemma and for every $x,y$ that are contained in the same element of $\tilde  \sB_n$, the orbit of $y$ from $-t^-_y$ to $t^+_y$ is contained in the $\beta$-scaled tubular neighborhood of $x$, until the orbit of $x$ leaves $B_{r}(\sigma)$.	
\end{proposition}

The next lemma shows that the cardinality of $\sB_n$ grows exponentially in $n$:
\begin{lemma}\label{l.Bcard}
	There exists constants $c_1>0$ such that for  $L''=(L^{K_1}e)^{\dim M}$ and for every $n>n_0$, we have
	$$
	\# \tilde\sB_n\le c_1(L'')^{n}.
	$$ 	
\end{lemma}
\begin{proof}
	We write $c_0 = L_0/e$ and $L'=L^{K_1}e$. Then for each $n>n_0$, \eqref{e.rn} becomes
	$$
	r_{n}= c_0 \beta  \left(L'\right)^{-n},
	$$
	which means
	$$
	\Leb(B_{r_{n}/2}(x))\ge c L'^{ - n\cdot \dim M}
	$$
	for some constant $c>0$ depending on $\beta$ and the Riemannian metric. Since for $x\ne y\in E_n$ we have $B_{r_n/2}(x)\cap B_{r_n/2}(y)=\emptyset$, it follows that
	$$
	\#\tilde \sB_n\le \# E_n \le \frac{\Leb M}{c L'^{ - n\cdot \dim M}}.
	$$
	Then the lemma follows with $c_1=\Leb (M)/c$, and $L'' = L'^{\dim M} = (L^{K_1}e)^{\dim M}$.
\end{proof}

Now we write
$$
 \sB_n = \{\phi_{[0,1)}(\tilde{B}): \tilde{B}\in\tilde \sB_n\}.
$$
Then $ \sB_n$ is a partition of $C_n$ for each $n>n_0$.

Recall that the closure of the set
$$
O(\sigma) = \bigcup_{n>n_0}\bigcup_{x\in D_n} \phi_{[-t^-_x,t^+_x]}(x)\subset B_{r}(\sigma)\\
$$
contains a neighborhood of $\sigma$. We also define:
$$
B^-(\sigma) =O(\sigma)\cap  \exp_\sigma\left(\{v\in T_\sigma M: |v|\le\beta, |v^s|<|v^u|\}\right),\\
$$
$$
B^+(\sigma) = O(\sigma) \setminus\left(B_\sigma^-\cup \bigcup_{n>n_0}C_n\right).
$$
$B^\pm(\sigma)$ can be seen as the regions in $O(\sigma)$ that sit ``above'' and ``below'' the set $\bigcup_{n>n_0}C_n$, respectively.  One should note that $\sigma\in\Cl( B^-(\sigma))\cap \Cl(B^+(\sigma))$.

We then define the partition $\sA_\sigma$ as:
\begin{equation}\label{e.partition}
\sA_\sigma = \{B : B\in \sB_n \mbox{ for some }n>n_0\}\cup\{B^-(\sigma),B^+(\sigma),O(\sigma)^c\}.
\end{equation}
Then $\sA_\sigma$ is a countable partition of $M$ which refines $\sC_\sigma$. Note that the partition is constructed locally inside the neighborhood $O(\sigma)$ of $\sigma$, and treat the complement of this neighborhood as a single partition element.

\subsection{Finite entropy}
Next, we show that the metric entropy of $\sA_\sigma$ is uniformly bounded from above:

\begin{proposition}\label{p.Aentropy}
	There exists $H_2>0$ depending on $L$, such that for every invariant probability measure $\mu$, 
	we have $$H_\mu(\sA_\sigma)<H_2<\infty.$$
\end{proposition}
\begin{proof}
	We use the following inequality for the conditional entropy:
	$$
	H_\mu(\sA_\sigma)\le H_\mu(\sA_\sigma|\sC_\sigma) + H_\mu(\sC_\sigma).
	$$
	Note that the second term is bounded by $H_1$ due to Proposition~\ref{p.Centropy}. For the first term, recall that $\sA_\sigma$ is obtained by refining each element of $\sC_\sigma$ with the partition $ \sB_n$. In the mean time, Lemma~\ref{l.Bcard} gives
	$$
	\# \sB_n = \# \tilde\sB_n\le c_1(L'')^{n}.
	$$
	We have
	\begin{align*}
	H_\mu(\sA_\sigma|\sC_\sigma) &= -\sum_n \mu(C_n)\sum_{B\in \sB_n}  \mu_{C_n}(B)\log \mu_{C_n}(B)\\
	&\le \sum_n\mu(C_n)\log \left(\# \sB_n\right)\\
	&\le \sum_n(n\log L''+\log c_1)\mu(C_n)\\
	&\le \log c_1 + \log L''\sum_n n\mu(C_n)\\
	&\le \log c_1 + \frac{1}{K_0}\log L''<\infty,
	\end{align*}
	where the last line follows from~\eqref{e.sumC}.  Now the proposition holds with $H_2=\log c_1+\frac{1}{K_0}\log L'' + H_1$.
\end{proof}

\begin{remark}\label{r.uniform}
	Following Remark~\ref{r.uniformconstant}, we see that the constants $c_0, c_1, L', L''$ and $H_2$ can be made uniform for nearby $C^1$ vector fields.
\end{remark}

\begin{proof}[Proof of Theorem~\ref{m.A}]
	Item (I) and (III) follows from Lemma~\ref{l.Bcard}, while (II) is given by Lemma~\ref{l.diam}. (IV) is precisely Proposition~\ref{p.tubular.sing}, and (V) is proven as Proposition~\ref{p.Aentropy}.
	
	For the continuity of the partition $\sA_n$, note that (all the continuity is in Hausdorff topology and $C^1$ topology): (1) the cross sections $D_\sigma$ varies continuously for nearby $C^1$ vector fields; the same holds for each $D_n$; (2) for each $n$, the partition $\tilde\sB_n$ can be made continuous; in particular, this means that $\sB_n$ is continuous; (3) the (finitely many) neighborhoods $O(\sigma)$ varies continuously, therefore $C_{reg}$ varies continuously; (4) the finite partition $\sA_{reg}$ can be made continuous w.r.t. nearby flows.
	
	We conclude the proof of Theorem~\ref{m.A}.
\end{proof}

\section{The partition $\sA$}\label{s.5}
In this section we will prove Theorem~\ref{m.3}. We assume that $X$ is a $C^1$ vector field such that all the singularities of $X$ are hyperbolic; in particular, $X$ has only finitely many singularities.

\subsection{Near each singularity}
First, note that $N_0$ in Theorem~\ref{m.A} is given by~\eqref{e.L}, which is defined for the entire flow. The same can be said about $\beta_0$ in Proposition~\ref{p.tubular1}. On the other hand, the constants $\beta_1,L_0, L_1, K_, K_1, L', L'', c_0, c_1$ in both Theorem~\ref{m.C} and \ref{m.A} depends on the hyperbolicity of each singularity.  Since there are only finitely many singularities, such constants can be made uniform for the vector field $X$ (also robust in a $C^1$ neighborhood).

Now, we can fix some $L\ge N_0$, $\beta<\beta_0$,  $r<\beta_1$ and apply Theorem~\ref{m.A} to obtain a countable partition $\sA_\sigma$ for each $\sigma\in\Sing(X)$. Each partition $\sA_\sigma$ also comes with a set $O(\sigma)\subset B_r(\sigma)$, and $n_0^\sigma\in\NN$. We will write
$$
n_0 = \max_{\sigma\in \Sing(X)} n_0^\sigma.
$$

\subsection{Away from singularities}
The set
$$
C_{reg} =\Cl\left( M\setminus (\cup_{\sigma\in\Sing X} \Cl(O(\sigma))) \right)
$$
consists only of regular points of $X$. Furthermore, by (IV) of Theorem~\ref{m.C}, points in $C_{reg}$ satisfies
$$
d(x,\Sing(X)) \ge e^{-n_0}.
$$

For $L\ge N_0$ and $\beta<\beta_0$ as above, the set
\begin{equation}\label{e.reg}
B(x): = \phi_{(-\frac14,\frac14)}\left(N_x\left(\frac{\beta}{2L} \cdot |X(x)|\right)\right)
\end{equation}
is the $\frac{\beta}{2}$-scaled tubular neighborhood starting at $\phi_{-\frac14}(x)$, with length $\frac12$. Moreover, if $B(x)\cap B(y)\ne\emptyset$, then both $B(x)$ and $B(y)$ are contained in the $\beta$-scaled tubular neighborhood at $\phi_{-1/2}(z)$ with length one, for every $z\in B(x)\cap B(y)$ (the choice of $\phi_{-1/2}(z)$ makes $z$ the ``center'' of this tubular neighborhood).

The collection $\{B(x): x\in C_{reg}\}$ forms an open covering of $C_{reg}$. Since $C_{reg}$ is compact, we can take a finite sub-covering $\{B(x_i): i=1,\ldots,k\}$,  and obtain a finite partition of $C_{reg}$, whose elements are given by the intersection of elements in the sub-covering. We denote this partition by $\sA_{reg}$.
Then for each $A\in\sA_{reg}$, $\partial A$ consists of flow lines with bounded length, and the normal manifold $N_x$ at some regular point.\footnote{To obtain $\sA_{reg}$, one can follow the language of {\em Bowen-Sinai refinement} for Markov partition, See~\cite{B08, Si}. This set-theoretical procedure refines a finite open covering into a finite partition, without destroying the local product structure. In our case, the local product structure is given by the normal manifolds and the flow lines (and recall that $C_{reg}$ is uniformly away from singularities).}

\subsection{The partition $\sA$}
Finally, we define the partition $\sA$ as
$$
\sA = \sA_{reg}\vee\bigvee_{\sigma\in\Sing(X)}\sA_\sigma.
$$
Then we have:
\begin{proposition}\label{p.entropy}
	For any invariant probability measure $\mu$, $h_\mu(\phi_1, \sA)$ is finite.
\end{proposition}
\begin{proof}
	We have
	$$
	H_\mu(\sA)\le H_\mu(\sA_{reg})+\sum_{\sigma\in\Sing(X)}H_\mu(\sA_\sigma).
	$$
	The first term is finite since $\sA_{reg}$ is a finite partition. Each term in the second summation is finite, thanks to Proposition~\ref{p.Aentropy}; also note that the summation itself has only finitely many terms since $X$ has only finitely many singularities.
\end{proof}

Note that for given $x\in M$ and $y\in \sA^\infty(x)$, by Proposition~\ref{p.tubular.sing} and the construction at regular points by~\eqref{e.reg}, we see that the orbit of $y$ must stay in the $\beta$-scaled tubular neighborhood of the orbit of $x$ forever. This finishes the proof of Theorem~\ref{m.3}.

In fact, more can be said:  generally, given a regular point $x\in C_{reg}$, 
the map:
$$
P_x(y):\sA(x)\to N_x(\beta)
$$
which maps the point $y\in\sA(x)$ to the unique point of intersection $\{P_x(y)\} = \phi_{[-1,1]}(y)\cap  N_x(\beta)$ is well-defined, since the partition $\sA(x)$ is contained in a scaled tubular neighborhood of the orbit of $x$.

Note that if $y\in\sA^\infty(x)$, then we have $\phi_j(y)\in A(\phi_j(x))$ for every $j\in\ZZ$. In particular, if $j\in\ZZ$ satisfies $\phi_j(x)\in C_{reg}$, then the construction of $\sA$ at regular points means that $\phi_j(x),\phi_j(y)$ are in the same tubular neighborhood with length $1$, and
\begin{equation}\label{e.reg.partition}
P_{\phi_j(x)}(\phi_j(y))\in N_{\phi_j(x)}(\beta).
\end{equation}

On the other hand, if $j\in\ZZ$ is such that $\phi_j(x)\in O(\sigma)$ for some singularity $\sigma$, then Proposition~\ref{p.tubular.sing} states that the orbit of $y$ and the orbit of $x$ are in the same scaled tubular neighborhood, as long as they are both in $O(\sigma)$; moreover, $y\in\sA^\infty(x)$ guarantees that $\Orb(y)$ and $\Orb(x)$ must hit the same element of $ \sB_n$ at the same iterate. Furthermore,  $\Orb(y)$ and $\Orb(x)$  must enter and leave $O(\sigma)$ at the same iterates under the time-one map $\phi_1$. This is because, once $x$ leaves $B^+(\sigma)$, it enters the partition at a regular point, which is contained in a scaled tubular neighborhood with length $1$. Since $y\in\sA^\infty(x)$, $y$ must enter the same element at the same iterate. In particular, this means that $\Orb(y)$ and $\Orb(x)$ spend the same amount of time in $B^\pm(\sigma)$; however, we lose control (in the sense that~\eqref{e.reg.partition} may not hold) for the orbit segment in $B^\pm(\sigma)$, since we treat each of them  as a single partition element.

\section{A general result on the expansiveness w.r.t. a partition}\label{s.6}
In this section we will prove Theorem~\ref{m.tailestimate}, which gives a criterion for partitions whose entropy is equal to the metric entropy. To put our result in a more general context, let $f: M\to M$ be a homeomorphism, $\mu$ an invariant probability such that $h_\mu(f)<\infty$.
Let $\sA$ be a measurable partition such that $H_\mu(\sA)<\infty$.

Before we dive into the proof, let us make some remark regarding our notion of expansiveness w.r.t. a partition. Following Bowen~\cite{B72},  the infinite Bowen ball is defined by
$$B_{\infty,\vep}(x)=\bigcap_{n\in \mathbb{Z}} f^{-n} B_\vep(f^n(x)),$$
and the $\vep$-tail entropy at $x$ is defined  as
$$h_{tail}(f,x,\vep)=h_{top}(B_{\infty,\vep}(x),f).$$
The system $f$ is $\vep$-entropy expansive if $h_{tail}(f,x,\vep) = 0 $ for all $x$. A measure $\mu$ is called $\vep$-almost entropy expansive, if $h_{tail}(f,x,\vep) = 0$ for $\mu$ a.e. $x$. Bowen proved that if $f$ is $\vep$-entropy expansive, then every finite partition $\sA$ with $\diam\sA<\vep$ satisfies
$$
h_\mu(f) = h_\mu(f,\sA)
$$
for every invariant measure $\mu$. On the other hand, it is proven in~\cite{LVY} that $f$ is $\vep$-entropy expansive if and only if every $f$ invariant measure $\mu$ is $\vep$-almost entropy expansive.

For any $x\in M$, recall that the $\infty$ $\sA$-ball is defined
$$\sA^\infty(x)=\bigcap_{n\in \mathbb{Z}} f^{-n} \sA(f^n(x)),$$
and the $\sA$-tail entropy of $x\in M$ is given by
$$h_{tail}(f,x,\sA)=h_{top}(\sA^\infty(x),f).$$
In other words, we are replacing the geometric balls $B_\vep$ in Bowen's definition of tail entropy by ``partition balls'' $\sA(x)$. Similarly, $\mu$ being $\sA$-expansive can be seen as the equivalence of $\vep$-almost expansiveness defined using the partition $\sA$.
Also note that if $\diam \sA<\vep$ and if $\mu$ is $\vep$-almost entropy expansive, then $\sA(x)\subset B_{\infty,\vep}(x)$ must have zero topological entropy. In other words, $\vep$-almost entropy expansive implies $\sA$-expansive when $\diam\sA<\vep$.

The key advantage of $\sA$-expansiveness is that, it allows us to obtain
$$
h_\mu(f) = h_\mu(f,\sA)
$$
for a particular measure $\mu$.

\begin{proof}[Proof of Theorem~\ref{m.tailestimate}]
	Since $M$ is finite dimensional, there is $m$ determined by the dimension of $M$, such that we can take finite partition
	$\sB=\{B_1,\cdots, B_m\}$ of the ambient manifold with arbitrarily small diameter, such that each point $x\in M$
	lies in at most $m$ elements of $\overline{\sB}=\{\Cl(B_i)\}_{i=1,\cdots,m}$.
	
	For any $E\subset M$, let
	$$F(E,\sB)=\{B\in \sB: B\cap E \neq \emptyset\}.$$
	
	Denote by $r_n(\delta,E)$ the minimal cardinality of $(n,\delta)$-spanning sets on $E$. The next lemma is due to Bowen:
	
	\begin{lemma}\cite{B72}
		Let $\overline{\sB}=\{\Cl(B_i)\}_{i=1,\cdots,m}$ be a compact cover of $M$. There is a $\delta>0$ such that
		$$\#(F(E,\overline{\sB}^n))\leq r_n(\delta,E) m^n$$
		for all $E\subset M$ and $n\geq 0$.
	\end{lemma}
	
	Let us continue the proof. We have	
	$$h_\mu(\sB)\leq h_\mu(\sB\bigvee \sA)=\lim\frac{1}{n}H_\mu(\vee_{i=0}^{n-1} f^{-i}(\sA)\bigvee \vee_{i=0}^{n-1} f^{-i}(\sB))$$
	$$\leq \limsup \frac{1}{n}[H_\mu(\vee_{i=0}^{n-1}f^{-i}(\sA))+H_\mu(\vee_{i=0}^{n-1}f^{-i}(\sB)\mid \vee_{i=0}^{n-1}f^{-i}(\sA))$$
	$$=h_\mu(\sA)+ \limsup \frac{1}{n}H_\mu(\vee_{i=0}^{n-1}f^{-i}(\sB)\mid \vee_{i=0}^{n-1}f^{-i}(\sA))$$
	
	Observe that
	$$H_\mu(\vee_{i=0}^{n-1}f^{-i}(\sB)\mid \vee_{i=0}^{n-1}f^{-i}(\sA))\leq \sum_{i=0}^{n-1} H_\mu(f^{-i}(\sB)\mid \vee_{j=0}^{n-1}f^{-j}(\sA))$$
	
	Fix $n_0>0$, for $i\geq n_0$.
	Because
	$$H_\mu(f^{-i}(\sB)\mid \vee_{j=0}^{n-1}f^{-i}(\sA))=H_\mu(\sB\mid \vee_{j=-i}^{n-i-1}f^{-i}(\sA))\leq H_\mu(\sB\mid \vee_{j=-n_0}^{n-i-1}f^{-i}(\sA))$$
	is decreasing to $H_\mu(\sB\mid \vee_{j=-n_0}^{\infty}f^{-j}(\sA))$, and because for $0\leq i <n_0$,
	$$H_\mu(f^{-i}(\sB)\mid \vee_{j=0}^{n-1}f^{-i}(\sA))\leq H_\mu(f^{-i}(\sB))=H_\mu(\sB),$$
	we have
	$$h_\mu(\sB)\leq h_\mu(\sA)+H_\mu(\sB\mid\vee_{j=-n_0}^{\infty}f^{-j}(\sA)).$$
	Since the above inequality holds for any $n_0$, and because $H_\mu(\sB\vee_{j=-n_0}^{\infty}f^{-j}(\sA))\searrow H_\mu(\sB\mid \vee_{-\infty}^\infty f^{-j}\sA)$.
	
	\begin{equation}\label{eq.onesteptail}
	h_\mu(\sB)\leq h_\mu(\sA)+H_\mu(\sB\mid \vee_{j=-\infty}^{\infty}f^{-j}(\sA)).
	\end{equation}
	
	Now for any $n>0$ we consider $F=f^n$, $\sB^n=\vee_{i=0}^{n-1} f^{-i}(\sB)$ instead of $\sB$ and $\sA^n=\vee_{i=0}^{n-1} f^{-i}(\sA)$,
	we have
	$$h_\mu(f^n,\sB^n)\leq h_\mu(f^n,\sA^n)+H_\mu(\sB^n\mid \vee_{j=-\infty}^{\infty}f^{-j}(\sA)).$$
	Because
	$h_\mu(f^n,\sB^n)=nh_\mu(f,\sB)$ and $h_\mu(f^n,\sA^n)=nh_\mu(f,\sA)$, we have
	$$h_\mu(f,\sB)\leq h_\mu(f,\sA)+\lim \frac{1}{n}H_\mu(\sB^n\mid \sA^\infty)$$
	$$\leq h_\mu(f,\sA)+\lim \frac{1}{n}\int \log \#(F(\sA^\infty(x),\overline{\sB}^n))d\mu(x)$$
	$$\leq h_\mu(f,\sA)+\lim \int \frac{1}{n} (\log r_n(\delta,\sA^\infty(x))+ nm) d\mu(x).$$
	Since $\frac{1}{n} \log r_n(\delta,\sA^\infty)\leq r_1(\delta, M)$, by dominate convergence,
	$$h_\mu(f,\sB)\le h_\mu(f,\sA)+\int h_{top}(\sA^\infty(x),f)d\mu(x)+m=h_\mu(f,\sA)+m.$$
	
	To get rid of $m$, for each $n>0$ we take $f^n$, $\bigvee_{i=0}^{n-1}f^{-i}(\sA)$ and $\sB$. Then we have
	$h_\mu(f^n,\sB)\leq h_\mu(f^n,\bigvee_{i=0}^{n-1}f^{-i}(\sA))+m$.
	Taking the diameter of $\sB$ converging to $0$, we have
	$h_\mu(f^n)\leq h_\mu(f^n,\bigvee_{i=0}^{n-1}f^{-i}(\sA))+m$,
	thus $$h_\mu(f)\leq h_\mu(f,\sA)+m/n,$$
	let $n\to \infty$, we finish the proof.
\end{proof}

\section{Application 1: entropy theory for flows away from tangencies}\label{s.7}

The rest of this paper is devoted to the entropy theory for star flows and flows away from homoclinic tangencies, using the partition $\sA$. In this section, we will show that the partition $\sA$ can be used to compute the metric entropy for any invariant measure.

The key idea of this proof is to relate the images of $\sA^\infty$ with a family of one-dimensional curves, whose length are well-controlled. For this purpose, we use an argument similar to~\cite{LVY}. However, we will see that the argument here is much more involved. This is because, in~\cite{LVY} when one considers a diffeomorphism away from tangencies, there exists a dominated splitting on the tangent bundle given by \cite{W04}. As we have discussed, such splitting controls a neighborhood of the invariant set with uniform size. However this is not the case for singular flows. As we will see below, the fake foliations are only defined for the scaled linear Poincar\'e flow; as a result, the size of such foliation is exponentially small when the orbit approaches a singularity.

\subsection{Fake foliations}
The following lemma is borrowed from \cite[Lemma 3.3]{LVY} (see also \cite[Proposition 3.1]{BW}), which shows that given a dominated splitting, one can always construct local fake foliations. Moreover, these fake foliations have local product structure, and this structure is preserved as long as they stay in a neighborhood.

\begin{lemma}\label{l.fakefoliation}
	Let $K$ be a compact invariant set of a diffeomorphism $f$. Suppose $K$ admits a dominated splitting $T_K M = E^1 \oplus E^2 \oplus E^3$. Then there are $\rho > r_0 > 0$, such that the neighborhood $B_\rho(x)$ of every $x \in K $ admits foliations $\cF^1_x, \cF^2_x, \cF^3_x, \cF^{12}_x$ and $\cF^{23}_x$, such that for every $y \in B_{r_0}(x)$ and $* \in \{1, 2, 3, 12, 23\}:$
	\begin{enumerate}[label=(\roman*)]
		\item $\cF^*_x(y)$ is $C^1$ and tangent to the respective cone.
		\item Forward and backward invariance: $f(\cF^*_x(y, r_0)) \subset \cF^*_{f(x)}(f(y))$, and\\ $f^{-1}(\cF^*_x(y, r_0)) \subset \cF^*_{f^{-1}(x)}(f^{-1}(y))$.
		\item $\cF^1_x$ and $\cF^2_x$ sub-foliate $\cF^{12}_x$; $\cF^2_x$ and $\cF^3_x$ sub-foliate $\cF^{23}_x$.
	\end{enumerate}
\end{lemma}

Note that such foliations are constructed locally. In particular, following their construction, one can show that if there is a dominated splitting $E_1\oplus E_2$ on the normal bundle $\cN_\Lambda$ for the scaled linear Poincar\'e flow $\psi^*_t$, which can be extended to a neighborhood by Proposition~\ref{p.tubular2}, then near every $x\in\Lambda$ one has fake foliations $\cF^i$ on $N_x$, and tangent to $E_i$, $i=1,2$. Furthermore, the size of such foliations are at least $r_0$ after scaling with the flow speed at $x$. In other words, for every regular point $x\in\Lambda$, there is fake foliation with size $r_0|X(x)|$.

From now on, we will assume that $\beta<\frac12\min\{r_0,\beta_0\}$. This makes the size of the $\beta$-scaled tubular neighborhood less than the size of the fake foliation at every point.

\subsection{Control the tail entropy of $\sA^\infty$}

Recall that the constant $N_0$ is defined by~\eqref{e.L}, and $K_0, K_1$ are the constants in Theorem~\ref{m.C}. 
Below, we will prove that if $L\ge N_0$, then the partition $\sA$ given by Theorem~\ref{m.3} for the constants $L$ and $\beta<\frac12\min\{r_0,\beta_0\}$ satisfies Theorem~\ref{m.tangency}.

The main result of this section is the following:

\begin{proposition}\label{p.Atailentropy}
Let $X\in \sX^1(M)\setminus \Cl(\cT)$ be a $C^1$ vector field such that all the singularities of $X$ are hyperbolic, and $\sA$ be the partition given by Theorem~\ref{m.3} for $\beta<\frac12\min\{r_0,\beta_0\}$ and $L\ge N_0$. Then $$h_{tail}(\phi_1,x,\sA)=0$$ for every invariant probability measure $\mu$ and $\mu$-a.e.\,$x$. 
In particular, we have
$$
h_\mu(\phi_1) =  h_\mu(\phi_1,\sA).
$$
\end{proposition}



It is proven in~\cite[Corollary 2.11]{GY} that for vector fields away from homoclinic tangencies,  there is a dominated splitting for both $\psi_t$ and $\psi^*_t$ on the normal bundle $\cN_\Lambda = \cup_{x\in\Lambda\setminus \Sing(X)} \cN_x$ over any invariant set $\Lambda$. Furthermore, following the proof of~\cite[Proposition 3.4]{LVY}, which uses the result of~\cite{W04} for diffeomorphisms away from homoclinic tangencies, we have the following lemma. 
Here the notation $\phi_{Y,t}$ and $\psi^*_{Y,t}$ represents the flow and the scaled linear Poincar\'e flow defined using the vector field $Y$.

\begin{lemma}\label{l.7.3}
	Let $X$ be a $C^1$ vector field away from tangencies. Then there exist $\lambda_0 > 0$, $J_0 \ge 1$, and a $C^1$ neighborhood $\cU_0$ of $X$ , such that, given any vector field $Y \in \cU_0$, the support of any ergodic $Y$-invariant measure $\mu$ admits an $L_0$-dominated splitting for both $\psi_t$ and $\psi_t^*$ over the normal bundle:
	$$\cN_{\supp \mu} = E^1 \oplus E^2 \oplus E^3\mbox{,	with }\dim(E^2)\le 1,$$
	and, for $\mu$-almost every point $x$, we have
	$$
	\lim_{n\to\infty}\frac1n\sum^n_{i=1}\log \|\psi^*_{Y,J_0}	| E^1_{\phi_{Y,?iL_0}(x)}\| \le ?\lambda_0,\mbox{ and }
	$$
	$$
	\lim_{n\to\infty}\frac1n\sum^n_{i=1}\log \|\psi^*_{Y,-J_0}	| E^3_{\phi_{Y,iL_0}(x)}\| \le ?\lambda_0.
	$$
\end{lemma}

\begin{remark}\label{r.7.3}
	The proof of the previous lemma exploits the fact that if $f$ is away from tangencies, then every periodic point of nearby diffeomorphism $g$ can have at most one eigenvalue with modulus one, which has to be real and has multiplicity one (if such eigenvalue exists). As a result, the constant $\lambda_0>0$, which is given by~\cite[Lemma 3.6]{W02}, can be made arbitrarily close to $0$.
\end{remark}

From now on, to simplify notation, we will fix any $Y\in\cU_0$ where $\cU_0$ is given by Lemma~\ref{l.7.3}, and write $g=\phi_{Y,1}$ for the time-one map of $Y$.
Following the proof of Theorem~3.1 in~\cite{LVY}, we see that for $\mu$ almost every $x$, the projection of $\sA^\infty(x)$ and its image along the flow to the normal manifold $N_x(\beta)= \exp_x(\cN_x(\beta))$ must be contained in the fake foliation tangent to $E^2$. To be more precise, for $\mu$ almost every point $x$, the map:
$$
P_x(y):\sA(x)\to N_x(\beta)
$$
which projects $\sA(x)$ to the normal manifold at $x$ along the flow must satisfy
\begin{equation}\label{e.22}
P_{g^j(x)}\circ g^j(\sA^\infty(x))\subset \cF^2_{g^j(x)}(g^j(x), r_1), \forall j\in\ZZ,
\end{equation}
where $\cF^2_\cdot(\cdot,r_1)$ is the fake foliation in $N_x(\beta)$ associated to the dominated splitting $E^1\oplus E^2\oplus E^3$ on $\cN_{\supp\mu}$, given by Lemma~\ref{l.fakefoliation}. 


Observe that in the case $\dim E^2 = 0$, there is nothing to prove since $\cF^2$ reduces to a point. In the case $\dim E^2 = 1$, the relation in~\eqref{e.22} significantly improves~\eqref{e.reg.partition}: the projection of $\sA^\infty$ to the normal manifolds of $g^j(x)$ is, in fact, contained in a family of one-dimensional curves with bounded length.
This in particular shows that $\sA^\infty(x)$ is contained in a two-dimensional strip, which is the image of the one-dimensional curves on $N_{g^j(x)}$ under the flow. This invites us to give the following general mechanism for a set to have zero topological entropy:

\begin{definition}
	We say that a set $A$ is {\em $(Y,\beta)$-shadowed} by a family of one-dimensional compact curves $\{I_j\}_{j\in\ZZ}$, if for every $j\in \ZZ$, $g^j(A)\subset \phi_{[-1,1]}(I_j)$.

\end{definition}

Note that we do not require $I_j$ to be contained in a tubular neighborhood of length $1$ at some point. Such requirement is only possible near regular points in $C_{reg}$, as we lose control in the region $B^\pm(\sigma)$.

\begin{proposition}\label{p.shadow.entropy}
	Let $A$ be a set that is $(Y,\beta)$-shadowed by $\{I_j\}_{j\in\ZZ}$, for $\beta<\frac12\min\{r_0,\beta_0\}$. If
	$$\lim_{n\to\infty}\frac1n\log \left(\sum_{j=-n}^n \length(I_j)\right) = 0,$$ then we have
	$
	h_{top}(A,\phi_1) = 0.
	$
\end{proposition}
The proof of this proposition is left to the appendix. We continue the proof of Theorem~\ref{m.tangency}.

Below we will construct the family of one-dimensional curves $\{I_j\}$ that $(Y,\beta)$-shadows $\sA^\infty(x)$, and control the length of $I_j$.

For each singularity $\sigma$, $n>n_0$ and $x\in B\in\sB_n$, we write $x^D$ for the unique point on $D_n$ such that $x = \phi_{a}(x^D)$ for some $a\in[0,1)$. We also define $\hat I(x)$ for the connected component of $\cF^2_x(x)\cap B$ that contains $x$, and $I(x)$ for the image of $\hat I(x)$ under the flow to $D_n$ (in fact, pre-image since $\hat I(x) \in \phi_{[0,1)}(D_n)$). It then follows that
$$
x^D\in I(x)\subset \sB(x^D).
$$

The following lemma gives a natural selection of $I_j$ near each singularity

\begin{lemma}\label{l.epsilon}
	There exists $C>0$, $\tilde\lambda>1$, such that For every $\sigma\in\Sing(Y)$, every $n>n_0$ and $x\in\sB_n$, we have
	$$
	\length(g^j(I(x)))\le  C\tilde\lambda^{-(t^+_x-j)}, \mbox{ for every } j\in [0, t^+_x],
	$$
	and
	$$
	\length(g^j(I(x)))\le  C\tilde\lambda^{-(t^-_x+j)}, \mbox{ for every } j\in [-t^-_x,0].
	$$
\end{lemma}

\begin{proof}
	We only need to consider the case $j\ge0$. The case $j\le0$ follows by considering the vector field $-Y$.
	
	First, recall that Lemma~\ref{l.diam} gives the estimate on the length of $I(x)$ as:
	$$
	\length(I(x)) \le \beta \frac{L_0}{e} (L^{K_1}e)^{-n}.
	$$
	Also recall that $\lambda'>1$ in Lemma~\ref{l.tx} is such that $\|Dg\mid_{B_r(x)}\|\le \lambda'$, and $K_0$ is chosen to be $\frac{1}{2\log \lambda'}$. Then we see that:
	\begin{align*}
	\length(g^j(I(x))) \le \beta \frac{L_0}{e} (L^{K_1}e)^{-n} \cdot \lambda'^j.
	\end{align*}
	
	We set
	$$
	\tilde\lambda : = (L^{K_1}e)^\frac{1}{K_0} >1.
 	$$
 	Our choice of $L\ge N_0$ guarantees that $\lambda'\le \tilde\lambda$. 	
 	Also note that $t^+_x \in [K_0n, K_1n]$ by Lemma~\ref{l.tx}.
 	
 	As a result, we obtain
 	$$
 	\length(g^j(I(x))) \le \beta \frac{L_0}{e} \tilde\lambda^{-K_0n}\tilde\lambda^{j}\le \beta \frac{L_0}{e} \tilde\lambda^{-t^+_x}\tilde\lambda^{j},
 	$$
 	as required.
\end{proof}

\begin{proof}[Proof of Proposition~\ref{p.Atailentropy}]
	For each $\sA^\infty(x)$, we will only construct the family of one-dimensional curves $\{I_j\}$ for $j\ge 0$. The case $j\le 0$ can be done using the same argument on the flow $-Y$.
	
	For each $j\ge 0$, we consider two cases:
	
	\noindent Case 1. $g^j(x)\in C_{reg}$. In this case, we take $I_j$ to be the connected component of $$\cF^2_{g^j(x)}(g^j(x), r_1)\cap N_{g^j(x)}(\beta|Y(g^j(x))|)$$ that contains $g^j(x)$. Then $I_j$ is in the $\beta$-scaled tubular neighborhood of $g^j(x)$, and $g^j(\sA^\infty(x))\subset \phi_{[-1,1]}(I_j)$ by~\eqref{e.22}. Note that in this case, we have $\length(I_j)\le r_1$.
	
	\noindent Case 2. $g^j(x)\in O(\sigma)$ for some $\sigma\in \Sing(Y)$. Due to the construction inside $O(\sigma)$, there is $n>n_0$ and $j'\in\NN$ such that
	$$
	g^{j'}(x)\in C_n, \,\, |j-j'|\le K_1n.
	$$
	Then we take $I_{j'} = I(g^{j'}(x))$, and $I_j = g^{j-j'}(I_{j'})$. In other words, we mark the nearest $j'$ such that the point $g^{j'}(x)$ is in the base $\cup C_n$, and define $I_{j'}$ to be the projection of $\cF^2(g^{j'}(x))$ to $D_n$, and iteration $I_{j'}$ to obtain $I_j$. This construction is consistent as long as the orbit of $x$ remains in $O(\sigma)$. 

	Then it is straight forward to verify that $\sA^\infty(x)$ is  $(Y,\beta)$-shadowed by $\{I_j\}$. To control the total length, for each $n>0$ we parse the orbit segment from $0$ to $n$ into:
	$$
	0 \le n_1< n_1' < n_2 < n_2' <\ldots < n_m\le n,
	$$
	where $n_i$ is the $i$th times where the orbit of $x$ enters $O(\sigma)$ for some $\sigma\in\Sing(Y)$, and $n_i'$ is the $i$th time that the orbit leaves $O(\sigma)$. For convenience we set  $n_0' = 0$ and $n_m' = n$.\footnote{That is, if $g^n(x)\in O(\sigma)$. If instead we have $g^n(x)\in C_{reg}$, then we have $n_m' < n$ and let $n_{m+1} = n$.}
	
	Now we write
	\begin{align*}
	\sum_{j=0}^n \length(I_j)\le& \sum_{j= 1}^m\sum_{k=n_{j-1}'}^{n_j-1}\length(I_k) + \sum_{j= 1}^m\sum_{k=n_{j}}^{n_j'}\length(I_k).
	\end{align*}
	Observe that first summation is taken along the orbit segment that is in $C_{reg}$; as a result, each term is bounded by $r_1$. As for the second sum, by Lemma~\ref{l.epsilon} there exists $\tilde{C}>0$, such that for each $j$, we have
	$$
	\sum_{k=n_{j}}^{n_j'}\length(I_k)\le \tilde{C}.
	$$
	Therefore, we obtain
	$$
	\sum_{j=0}^n \length(I_j)\le\sum_{j= 1}^m\sum_{k=n_{j-1}'}^{n_j-1}r_1 + \sum_{j= 1}^m \tilde{C}\le n(r_1+\tilde{C}).
	$$
	In particular, we have
	$$
	\frac1n\log \left(\sum_{j=0}^n \length(I_j)\right) \xrightarrow{n\to\infty} 0.
	$$
	By Proposition~\ref{p.shadow.entropy}, this shows that $h_{top}(\sA^\infty(x), g) = 0.$
\end{proof}

Now Theorem~\ref{m.tangency} follows from Proposition~\ref{p.Atailentropy} and Theorem~\ref{m.tailestimate}.

\begin{proof}[Proof of Corollary~\ref{mc.star1}]
	Let $X$ be a star vector field. We want to show that for every measure $\mu$ and $\mu$ almost every $x$, the set $\sA^\infty(x)$ reduces to a flow segment. 	
	Since $X$ is star, every critical element of $X$ is hyperbolic. Therefore it suffices to consider those $\mu$ that are non-trivial, that is, $\mu$ is not supported on a singularity or a periodic orbit.
	
	By~\cite[Theorem 5.6]{GSW}, every ergodic invariant measure $\mu$ is hyperbolic. In fact, following the proof of~\cite[Theorem 5.6]{GSW}, we see that there exists $\eta>0$, such that every non-trivial measure $\mu$ does not have any Lyapunov exponent in $(-\eta,\eta)$. By Remark~\ref{r.7.3}, we may take $\lambda_0<\eta$ in Lemma~\ref{l.7.3}, making the bundle $E^2$ trivial. 
	This in particular means that $\sA^\infty(x)$ reduces to a flow segment containing $x$.
\end{proof}

\section{Application 2: upper semi-continuity}\label{s.8}

In this section we will prove Theorem~\ref{m.continuous}. The main result here is Theorem~\ref{t.finitepartitionentropy}, which estimates the drop in the metric entropy when approximating $\sA$ with a finite partition.

Let $X_n$ be a sequence of $C^1$ vector fields, approaching $X$ in $C^1$ topology. Let $\mu_n$ be a sequence of probability measures, invariant under $T_n$. We assume that $\mu_n\xrightarrow{weak^*}\mu$ where $\mu$ is an invariant probability measure of $X$

To simplify notation, we will write $X=X_0$ and $\mu = \mu_0$. We will make the standard assumption that the sequence $\{X_n\}$ is contained in the $C^1$ neighborhood of $X$ described in Theorem~\ref{m.3}.
Let $\sA_n$ be the partition defined in Section~\ref{s.5} for $L=N_0,\beta<\frac12\min\{r_0,\beta_0\}$ using the flow $X_n$, for $n=0,1,\ldots$. Then Theorem~\ref{m.3} shows that $\sA_n\to\sA$. We denote by $\sA_{n,\sigma}, \sA_{n,reg},  \sC_{n,\sigma}, \sB_{n,m}$ for the partitions defined in Theorem~\ref{m.C} and~\ref{m.A} for the flow $X_n$. 
Note that the index $\sigma$ refers to the continuation of $\sigma$ for the flow $X_n$, and in general is different from $\sigma\in \Sing(X)$ itself.

By Proposition~\ref{p.Atailentropy}, we have
$$
h_{\mu_n}(X_n) = h_{\mu_n}(\phi_{X_n,1},\sA_n) , n=0,1,\ldots,
$$
where $\phi_{X,1}$ is the time-one map of the flow $X$.

The key idea in the proof of Theorem~\ref{m.continuous} is that, we need to obtain a {\em finite} partition by glueing certain elements of $\sA_n$ together. To this end, we fix some $N>n_0$ and define:
\begin{align*}
\sA_{n, \sigma,N} = &\{B: B\in \sB_{n,m} \mbox{ for some }n_0<m\le N\}\\&\cup\{B^-(\sigma),B^+(\sigma),O(\sigma)^c\}\cup\left\{\cup_{m>N} C_{n,m}\right\}.
\end{align*}
In other words, $\sA_{n, \sigma,N}$ is a finite partition obtained by taking the partition $\sA_{n, \sigma}$ defined by~\eqref{e.partition} for the flow $X_n$, and glueing all the partition elements of $ \sB_{n,k}$, $k>N$, into one set (the last term). See Figure~\ref{f.finitepartition}.

For each $\sigma$, we have thus obtained a sequence of finite partitions $\{\sA_{n,\sigma,N}\}_{n\ge 0}$. Next, we write, for $n=0,1,\ldots,$
\begin{equation}
\sA_{n,N} = \sA_{n,reg}\vee\bigvee_{\sigma\in\Sing(X_n)}\sA_{n,\sigma,N}.
\end{equation}
Then for each $n$, $\sA_{n,N}$ is a {\em finite} partition obtained by glueing all the partition elements of $\sA_n$ near each singularity into one element. It is clear that $\sA_n$ refines $\sA_{n,N}$ for every $N>n_0, n=0,1,\ldots$.

Next, we define, for each $\sigma$ (and its continuation):
\begin{equation}
	O^N(\sigma) =  \bigcup_{n>N}\bigcup_{x\in D_n} \phi_{[-t^-_x,t^+_x]}(x).
\end{equation}
Clearly we have $O^N(\sigma)\subset O(\sigma)$ for each $N>n_0$, and $\cap_{k>N} O^k(\sigma) = \emptyset$. Also note that
$$
\bigcap_{k>N} \Cl(O^k(\sigma)) = \sigma\cup W_\loc^s(\sigma )\cup W_\loc^u(\sigma),
$$
where $W^{s/u}_\loc(\sigma)$ is the stable and the unstable manifold of $\sigma$ contained in $B_r(\sigma)$.

\begin{figure}
	\centering
	\def\svgwidth{\columnwidth}
	\includegraphics[scale=1.2]{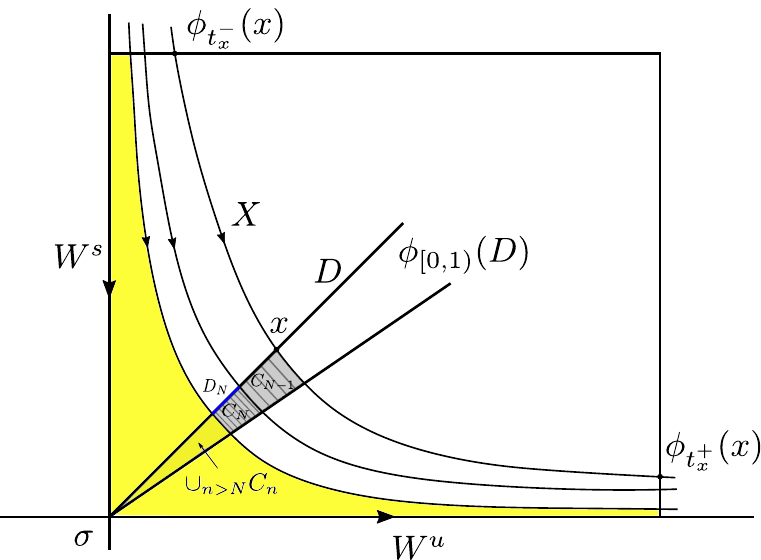}
	\caption{The finite partition $\sA_{\sigma,N}$ for the original flow $X$. The yellow region is $O^N(\sigma)$.}
	\label{f.finitepartition}
\end{figure}
The next theorem controls the loss of the metric entropy during this glueing process. In particular, it shows that the loss of the metric entropy is proportional to the measure of $O^N(\sigma)$:

\begin{theorem}\label{t.finitepartitionentropy}
	Let $X$ be a $C^1$ vector fields, with all the singularities hyperbolic. Let $\sA$ be the partition given by Theorem~\ref{m.3} for the constants $L=N_0$. Then there exists a constant $L_2>0$, such that for any invariant probability measure $\mu$ of $X$ and every $N>n_0$, we have
	$$
	h_{\mu}(\phi_{1},\sA)-L_2\sum_{\sigma\in\Sing(X)}\mu(O^N(\sigma)) - u_{X,\mu}(N)\le h_{\mu}(\phi_{1},\sA_{0,N})\le h_{\mu}(\phi_{1},\sA),
	$$
	for every $n$. Here $u_{X,\mu}(N)$ is a function of $N$ that converges to zero as $N\to\infty$, uniformly in $\mu$ and in a neighborhood of $X$.
	
	Furthermore, $L_2$ can be made uniform for nearby $C^1$ vector fields.
\end{theorem}
\begin{proof}
	The second inequality follows from the fact that $\sA_{0,N}$ is coarser than $\sA$. To obtain the first inequality, we write: 
	
	\begin{align*}
	&h_{\mu}(\phi_{1},\sA)-h_{\mu}(\phi_{1},\sA_{0,N})\\
	=&\lim_{k\to\infty} H_{\mu}(\sA\,\big|\bigvee_{j=1}^{k}\phi_{1}^{-j}\sA)-H_{\mu}(\sA_{0,N}\big|\bigvee_{j=1}^{k}\phi_{1}^{-j}\sA_{0,N})\\
	\le&\lim_k\Bigg(H_{\mu}(\sA\,\big|\sA_{0,N}) +H_\mu(\sA_{0,N}\big|\bigvee_{j=1}^{k}\phi_{1}^{-j}\sA_{0,N}) + H(\bigvee_{j=1}^{k}\phi_{1}^{-j}\sA_{0,N}\Big|\bigvee_{j=1}^{k}\phi_{1}^{-j}\sA)\\& -H_{\mu}(\sA_{0,N}\big|\bigvee_{j=1}^{k}\phi_{1}^{-j}\sA_{0,N})\Bigg).
	\end{align*}
	Note that the second term is cancelled with the forth, and the third term is zero since $\bigvee_{j=1}^{k}\phi_{1}^{-j}\sA$ is a refinement of $\bigvee_{j=1}^{k}\phi_{1}^{-j}\sA_{0,N}$. The only remaining term,which is the first term, does not depend on $k$. We thus conclude that
	\begin{equation}\label{e.26}
	h_{\mu}(\phi_{1},\sA)-h_{\mu}(\phi_{1},\sA_{0,N})\le H_{\mu}(\sA\,\big|\sA_{0,N}).
	\end{equation}
	It remains to show that
	$$
	H_{\mu}(\sA\,\big|\sA_{0,N})\le L_2\sum_{\sigma\in\Sing(X)}\mu(O^N(\sigma)) + u_{X,\mu}(N)
	$$
	for some $L_2>0$ and some function $u_{X,\mu}(N)$, which holds if we can prove that
	\begin{equation}\label{e.27}
		H_{\mu}(\sA_\sigma\,\big|\sA_{0,\sigma,N})\le L_2\mu(O^N(\sigma))+u_{X,\mu,\sigma}(N),
	\end{equation}
	for some function $u_{X,\mu,\sigma}(N)$ that converges to zero as $N\to \infty$, uniformly in $\mu, X$ and $\sigma$. Since $\Sing(X)$ only contains finitely many singularities, we then set
	$$
	u_{X,\mu}(N) =\sum_{\sigma\in\Sing(X)}u_{X,\mu,\sigma}(N)
	$$
	which also goes to zero uniformly.
	
	It remains to prove~\eqref{e.27}. The proof is, in fact, hidden in the proof of Proposition~\ref{p.Aentropy}. We define
	\begin{align*}
	\tilde \sA_{0,\sigma,N}  = &\{B: B\in \sB_{m} \mbox{ for some }n_0<m\le N\}\\&\cup\{B^-(\sigma),B^+(\sigma),O(\sigma)^c\}\cup\left\{C_{m}: m>N\right\}.
	\end{align*}
	Then $\tilde \sA_{0,\sigma,N}$ is a countable partition, obtained by glueing {\em each} $\sB_m$
	with $m>N$ into $C_m$. Then $\sA_{0,\sigma,N}$ can be seen as glueing {\em all the} $C_m, m>N$ into one element $C^{N} = \cup_{k>N}C_{k}$. We immediately see that $\sA_\sigma$ refines $\tilde \sA_{0,\sigma,N}$, while the latter refines $\sA_{0,\sigma,N}$.
	
	Now we write
	\begin{align*}
		H_{\mu}(\sA_\sigma\,\big|\sA_{0,\sigma,N})\le& H_{\mu}(\sA_\sigma\,\big|\tilde\sA_{0,\sigma,N}) + H_{\mu}(\tilde\sA_{0,\sigma,N} | \sA_{0,\sigma,N})\\
		=& I+II.
	\end{align*}
	
	First, note that all three partitions coincide outside $C^N$. Therefore, we can estimate $I$ as:
	\begin{align*}
	I&= H_{\mu}(\sA_\sigma\,\big|\tilde\sA_{0,\sigma,N}) \\
	&\le -\sum_{m>N}\mu(C_m) \sum_{B\in\sB_m} \mu_{C_m} (B)\log \mu_{C_m}(B)\\
	&\le \sum_{m>N}\mu(C_m)\log \left(\# \sB_m\right)\\
	&\le \sum_{m>N}(m\log L''+\log c_1)\mu(m)\\
	&\le \mu(C^N)\log c_1 + \log L''\sum_{m>N} m\mu(C_m)\\
	&\le \frac{\log c_1}{K_0N} +\log L''\sum_{m>N} \frac{1}{K_0}\mu\left(\bigcup_{x\in D_n}\phi_{[0,t^+_x]}(x)\right)\\
	&\le \frac{\log c_1}{K_0N} + \frac{\log L''}{K_0}\mu(O^N(\sigma)).
	\end{align*}
	Here we used Lemma~\ref{l.Bcard} for $\# \sB_m$,  Lemma~\ref{l.measure.triangle} for the measure of $C^N$, and~\eqref{e.sumC} to control $\sum_{m>N} m\mu(C_m)$.
	
	On the other hand, $II$ can be controlled as:
	\begin{align*}
	II=& H_{\mu}(\tilde\sA_{0,\sigma,N} | \sA_{0,\sigma,N}) \\
	\le& - \sum_{m>N} \mu (C_m)\left(\log \mu(C_m) - \log\mu(C^N)\right)\\
	\le& \, \mu(C^N)\log \mu(C^N) + \sum_{m>N} \mu (C_m)|\log \mu(C_m)|.
	\end{align*}
	Thanks to the uniform estimate on the measure of $\mu(C^N)$ by Lemma~\ref{l.measure.triangle}, we see that the first term goes to zero uniformly in $\mu$ and $X$.
	
	For the second term, we use Ma\~n\'e's proof of~\ref{l.finiteentropy} in~\cite{Ma81}. We write $a_n = \mu(C_n)$, and define the set
	$$
	\cG = \{n: a_n>e^{-n}\} = \{n: |\log a_n|<n\}.
	$$
	Then we have
	\begin{align*}
	\sum_{m>N} \mu (C_m)|\log \mu(C_m)| =&\sum_{m>N, m\in\cG} a_m|\log a_m| + \sum_{m>N, m\in\cG^c} a_m|\log a_m|\\
		\le& \sum_{m>N}m a_m +  \sum_{m>N, m\in\cG^c}\sqrt{e^{-m}} \cdot \sqrt{a_m}|\log a_m|.
	\end{align*}
	It is easy to see that the second term in the last line is of order $\cO(e^{-N/2})$ with the hidden constant uniformly bounded in $\mu$ and $X$. Therefore we have
	\begin{align*}
	\sum_{m>N} \mu (C_m)\log \mu(C_m)
	\le& \sum_{m>N}m \mu(C_m)+ \cO(e^{-N/2})\\
	\le &\sum_{m>N} \frac{1}{K_0}\mu\left(\bigcup_{x\in D_n}\phi_{[0,t^+_x]}(x)\right) + \cO(e^{-N/2})\\
	\le& \frac{1}{K_0}\mu(O^N(\sigma))+\cO(e^{-N/2}),
	\end{align*}
	where we used~\eqref{e.sumC} again to control the sum over $m\mu(C_m)$.

	Now we collect $I$, $II$ and obtain
	\begin{align*}
	H_{\mu}(\sA_\sigma\,\big|\sA_{0,\sigma,N})\le& \frac{\log c_1}{K_0N} + \frac{\log L''}{K_0}\mu(O^N(\sigma))+ \mu(C^N)\log \mu(C^N)\\&+\frac{1}{K_0}\mu(O^N(\sigma))+\cO(e^{-N/2}).
	\end{align*}
	In particular, \eqref{e.27} follows with $L_2=\frac{\log L''+1}{K_0}$, which is uniform in a $C^1$ neighborhood of $X$.
	 	
	With that we conclude the proof of Theorem~\ref{t.finitepartitionentropy}.
\end{proof}
As an immediate corollary, we have:

\begin{corollary}\label{c.1}
	Under the assumptions of Theorem~\ref{t.finitepartitionentropy}, if $\mu(\Sing(X)) = 0$, then for every $\vep>0$ we can take $N>n_0$ such that
	$$
	h_{\mu}(\phi_{1},\sA)-\vep\le h_{\mu}(\phi_{1},\sA_{0,N})\le h_{\mu}(\phi_{1},\sA),
	$$
\end{corollary}
\begin{proof}
	Observe that $\sum_{\sigma\in\Sing(X)}\mu(O^N(\sigma)) \le \sum_{\sigma\in\Sing(X)}\mu(\Cl(O^N(\sigma)))$, and
	$$\bigcap_{k>N} \Cl(O^k(\sigma))=\sigma\cup W_\loc^s(\sigma )\cup W_\loc^u(\sigma),$$ which has measure zero. So we can take $N>n_0$ large enough, such that $u_{X,\mu}(N)<\vep/2$ and $\sum_{\sigma\in\Sing(X)}\mu(O^N(\sigma)) <\vep/2$.	
\end{proof}

For each given $N$, we have obtained a sequence of finite partitions $\{\sA_{n,N}\}_{n=0}^\infty$. The next proposition is well known in the classical entropy theory. See for example~\cite{B08}.

\begin{proposition}\label{p.finitecontinuous}
	Under the assumptions of Theorem~\ref{m.continuous}, for each $N\in\NN$ large enough, we have
	$$
	\lim_{n\to\infty}h_{\mu_n}(\phi_{X_n,1},\sA_{n,N})\le h_{\mu_0}(\phi_{X_0,1},\sA_{0,N}).
	$$
\end{proposition}
Now Theorem~\ref{m.continuous} is a direct consequence of Theorem~\ref{t.finitepartitionentropy} and Proposition~\ref{p.finitecontinuous}, and the observation that for every $\vep>0$, one can take $N$ large enough such that $$\sum_{\sigma\in\Sing(X)}\mu(O^N(\sigma)) < \sum_{\sigma\in\Sing(X)}\mu(\Sing(X))+\vep.$$ The case $\mu(\Sing(X))=0$ follows from Corollary~\ref{c.1}. \qed

\appendix
\section{ Proof of Proposition~\ref{p.shadow.entropy}}
\begin{proof}
	First, note that if we had $g^j(A)\subset I$ instead of $g^j(A)\subset \phi_{[-1,1]}(I_j)$, then this proposition is immediate (in fact, this argument is already used in~\cite{LVY}). This is because for each $\vep>0$, the number of $\vep$-balls needed to cover $I_j$ is of the order $\cO(\frac{1}{\vep}\length(I_j))$. Also note that the set $A$ induces a natural order on each $I_j$. As a result, the sub-exponential growth of $\sum_{j=-n}^n \length(I_j)$ implies the sub-exponential growth of the cardinality of a $(\vep,n)$-spanning set.
	
	In the case $g^j(A)\subset \phi_{[-1,1]}(I_j)$, we define
	$$
	\tilde I_j = P_{x_{j}}(I_j).
	$$
	Then we have $\length(\tilde I_j) \le I_j$. The set $\phi_{[-2,2]}(\tilde I_j)$ contains $h^j(A)$. Furthermore, there exists a constant $C$ determined by the vector field $X$, such that $\phi_{[-2,2]}(\tilde I_j)$ can be covered by no more than $\frac{C}{\vep^2}\length(\tilde I_j)$ many $\vep$-balls. In the meantime, $A$ induces a natural order on each $\tilde{I_j}$. As a result, the minimal cardinality of a $(n,\vep)$-spanning set is bounded from above by
	$$
	\frac{C}{\vep^2}\sum_{j=-n}^n\length(\tilde I_j)\le 	\frac{C}{\vep^2}\sum_{j=-n}^n\length( I_j).
	$$
	This shows that
	$$
	h_{top}(A,g) \le \lim_{\vep\to 0}\lim_n\frac{1}{n}\log \left(\frac{C}{\vep^2}\sum_{j=-n}^n\length( I_j)\right)=0.
	$$
\end{proof}

\end{document}